\documentclass[3p,final]{elsarticle}
\usepackage{lineno,hyperref}
\modulolinenumbers[5]
\usepackage{graphicx}
\usepackage{ulem}
\usepackage{hyperref}
\hypersetup{
    colorlinks=true,
    linkcolor=blue,
    filecolor=magenta,      
    urlcolor=cyan,
}
\usepackage{amsmath}
\usepackage{amssymb}
\usepackage{latexsym}
\usepackage[english]{babel}
\usepackage{amsthm}
\usepackage{subfig}

\newtheorem{remark}{Remark}
\newtheorem{theorem}{Theorem}
\newtheorem{definition}{Definition}
\journal{TBA}

\usepackage{etoolbox}
\makeatletter
\patchcmd\ps@pprintTitle
  {\fi\fi\fi\fi}
  {\fi\fi\fi\fi
   \afterassignment\fix@elsarticle\let\@tempa}
  {}{\FailedToPatch}
\def\fix@elsarticle{\iffalse{\fi}\romannumeral-`0}
\makeatother

\makeatletter
\def\ps@pprintTitle{%
 \let\@oddhead\@empty
 \let\@evenhead\@empty
 \def\@oddfoot{}%
 \let\@evenfoot\@oddfoot}
\makeatother

\begin{document}

\begin{frontmatter}

\title{Analysis and simulation of a variational stabilization for the Helmholtz equation with noisy Cauchy data}
\tnotetext[mytitlenote]{This work was supported by the Faculty Research Awards Program (FRAP) at Florida A\&M University, under the project ``Approximation of a time-reversed reaction-diffusion system in cancer cell population dynamics” (007633).}

\author[mymainaddress]{Vo Anh Khoa}\corref{mycorrespondingauthor}
\cortext[mycorrespondingauthor]{Corresponding author}
\ead{anhkhoa.vo@famu.edu}

\author[mysecondaryaddress]{Nguyen Dat Thuc}

\author[mymainaddress]{Ajith Gunaratne}

\address[mymainaddress]{Department of Mathematics, Florida A\&M University, Tallahassee, FL 32307, USA}
\address[mysecondaryaddress]{Department of Mathematics and Computer Science, VNUHCM-University of Science, Ho Chi Minh City, Vietnam}

\begin{abstract}
This article considers a Cauchy problem of Helmholtz equations whose solution is well known to be exponentially unstable with respect to the inputs. In the framework of variational quasi-reversibility method, a Fourier truncation is applied to appropriately perturb the underlying problem, which allows us to obtain a stable approximate solution. The corresponding approximate problem is of a hyperbolic equation, which is also a crucial aspect of this approach. Error estimates between the approximate and true solutions are derived with respect to the noise level. From this analysis, the Lipschitz stability with respect to the noise level follows.  Some numerical examples are provided to see how our numerical algorithm works well.

\end{abstract}


\end{frontmatter}

\section{Introduction}

\subsection{Statement of the Cauchy problem}

In this work, we are concerned with the reconstruction problem of electromagnetic field from its knowledge on a part of boundary of the physical region $\Omega$. Here, $\Omega = (0,1)\times (0,1)$ is our computational domain of interest, but it can be extended easily to $(0,a_1)\times(0,a_2)$, where $a_1,a_2$ are two positive numbers. Often, the propagation of the electromagnetic wave field is governed by the
system of the Maxwell's equations for the electric field $\mathbf{E}=\mathbf{E}(x,y,t)$ and the magnetic field $\mathbf{B}=\mathbf{B}(x,y,t)$. Considering $\Omega$ as a homogeneous medium in a region with free currents and charges, this system can be reduced to the classical wave equations, cf. \cite{MaxBorn2019},
\begin{align}\label{wave}
	\frac{\partial^{2}\mathbf{E}}{\partial t^{2}}-c^{2}\Delta\mathbf{E}=0,\quad\frac{\partial^{2}\mathbf{B}}{\partial t^{2}}-c^{2}\Delta\mathbf{B}=0\quad \text{in }\Omega\times(0,T),
\end{align}
where $c>0$ is the speed of light and $T>0$ is the travel time. Consider the frequency $\omega>0$. For $\text{i}=\sqrt{-1}$, we take $\mathbf{E}(x,y,t) = e^{\text{i}\omega t}\mathbf{E}(x,y)$ and $\mathbf{B}(x,y,t) = e^{\text{i}\omega t}\mathbf{B}(x,y)$. Then, setting $k = \omega/c >0$, it follows from system \eqref{wave} that
\begin{align}\label{system1}
	\Delta \mathbf{E}(x,y) + k^2 \mathbf{E}(x,y) = 0,\quad \Delta \mathbf{B}(x,y) + k^2 \mathbf{B}(x,y) = 0,
\end{align}
which form a system of Helmholtz equations. Since system \eqref{system1} is uncoupled and linear with respect to each component of $\mathbf{E}(x,y)$ and $\mathbf{B}(x,y)$, it is pertinent to solve the following model:
\begin{align}\label{helmholtz}
	\Delta u\left(x,y\right)+k^{2}u\left(x,y\right)=0\quad\text{in }\Omega.
\end{align}
Note that in \eqref{system1}, $\mathbf{E}(x,y)$ and $\mathbf{B}(x,y)$ are complex-valued components, but it is sufficient to find a real-valued function $u=u(x,y)$ in \eqref{helmholtz}. Physically, fields vanish far from the axes and thus, we can assume that the electromagnetic field vanishes on the sides $\left\{ y=0\right\} $ and $\left\{ y=1\right\} $ of the computational domain $\Omega$. Mathematically, we consider
\begin{align}\label{boundary}
	u(x,0) = u(x,1) = 0\quad \text{for } x\in (0,1).
\end{align}
On the other hand, we assume to measure the electromagnetic Cauchy data at $\left\{ x=0\right\} $,
\begin{align}\label{boundary1}
	u(0,y) = u_0(y),\quad u_x(0,y) = u_1(y)\quad \text{for } y\in (0,1).
\end{align}

Cf. \cite{Tuan2017}, we remark that the second data (i.e. the Neumann data at $x=0$) in \eqref{boundary1} can be reduced to the zero boundary condition. In fact, let $U=U(x,y)$ be a solution to the following system:
\begin{align}\label{UU}
	\begin{cases}
		\Delta U\left(x,y\right)+k^{2}U\left(x,y\right)=0 & \text{in }\Omega,\\
		U\left(x,0\right)=U\left(x,1\right)=0 & \text{for }x\in\left(0,1\right),\\
		U_{x}\left(0,y\right)=u_{1}\left(y\right),U\left(1,y\right)=0 & \text{for }y\in\left(0,1\right).
	\end{cases}
\end{align}
Next, consider $V=V(x,y)$ as a solution to the following system:
\begin{align}\label{VV}
	\begin{cases}
		\Delta V\left(x,y\right)+k^{2}V\left(x,y\right)=0 & \text{in }\Omega,\\
		V\left(x,0\right)=V\left(x,1\right)=0 & \text{for }x\in\left(0,1\right),\\
		V\left(0,y\right)=u_{0}\left(y\right)-U\left(0,y\right),V_{x}\left(0,y\right)=0 & \text{for }y\in\left(0,1\right).
	\end{cases}
\end{align}
With \eqref{UU} and \eqref{VV}, it is clear that the solution $u$ to system \eqref{helmholtz}--\eqref{boundary1} can be computed via $u = U + V$. By \cite[Lemma 1]{Tuan2017}, we know that system \eqref{UU} is well-posed with $U$ in $H^2(\Omega)$ when $u_1 \in L^2(0,1)$ and thus, $U(0,y)$ exists in $H^2(0,1)$ by the embedding $H^2(\Omega)\subset C([0,1];H^2(0,1))$. Henceforth, from \eqref{VV}, instead of working on the Cauchy data \eqref{boundary1} we can assume that $u_1 = 0$ in \eqref{boundary1} in our analysis below.

Combining \eqref{helmholtz}, \eqref{boundary} and \eqref{boundary1} with $u_1 = 0$ forms our Cauchy problem for the Helmholtz equation. In this scenario, we want to reconstruct the whole wave field in $\Omega$ and especially, the field at the boundary $x=1$.

\begin{remark}
	Cf. the appendix of \cite{Klibanov2019a}, if the incident electric wave field has only one non-zero component, then the propagation of this component in a heterogeneous medium is governed equally well by a single Helmholtz equation. In other words, the Helmholtz equation may play an equal role as the Maxwell's system when the medium is no longer homogeneous as assumed above.
\end{remark}

\begin{remark}
	Mathematically, our numerical approach under investigation below can be extended easily to the  nonhomogeneous case of \eqref{helmholtz}--\eqref{boundary1}. In particular, our method can solve the following system of $u$:
	\[
	\begin{cases}
		\Delta u\left(x,y\right)+k^{2}u\left(x,y\right)=f\left(x,y\right) & \text{in }\Omega,\\
		u\left(x,0\right)=b_{0}\left(x\right),\quad u\left(x,1\right)=b_{1}\left(x\right) & \text{for }x\in\left(0,1\right),\\
		u\left(0,y\right)=u_{0}\left(y\right),\quad u_{x}\left(0,y\right)=u_{1}\left(y\right) & \text{for }y\in\left(0,1\right).
	\end{cases}
	\]
	In this context, our problem for $V$ remains the same as in \eqref{VV}. The problem for $U$ should then read as
	\begin{align*}
		\begin{cases}
			\Delta U\left(x,y\right)+k^{2}U\left(x,y\right)=f(x,y) & \text{in }\Omega,\\
			U\left(x,0\right)=b_0(x),\quad U\left(x,1\right)=b_1(x) & \text{for }x\in\left(0,1\right),\\
			U_{x}\left(0,y\right)=u_{1}\left(y\right),\quad U\left(1,y\right)=0 & \text{for }y\in\left(0,1\right),
		\end{cases}
	\end{align*}
	which is a well-posed boundary value problem.
\end{remark}

\subsection{Historical remarks and our goals}

The Cauchy problem for Helmholtz equations (as well as elliptic equations) is well-studied in the Inverse and Ill-posed Problems community. Suffering from the Hadamard instability, this problem is severely ill-posed as the degree of ill-posedness is infinite; see \cite{Karimi2022} for distinctive classes of ill-posed problems based on the degree of ill-posedness. To overcome the natural instability, there are many researches devoted to regularization of such a Cauchy problem. Those are essentially spectral-based and optimization-based methods. The existing literature on these two types of methods is huge. The spectral regularization method and its variants rely on suitable perturbation of the unbounded kernel involved in the explicit presentation of solution. The kernel can be stabilized by the perturbation of the original PDE or by the direct perturbation inside the kernel. The former perturbation may lead to the so-called PDE-based regularization method. The reader can be referred to the following fundamental works \cite{Leitao2000,Qian2008,Tuan2010a,Elden2009} and references cited therein for an overview of the spectral regularization method. The optimization-based regularization method is based on the construction of Tikhonov-like cost functionals involving the (strict) convexity; cf. e.g. \cite{Hao2000,Klibanov2015,Reinhardt1999}. The obtained minimizer is proved to approximate the true solution in a stable manner. It is necessary to mention here the works \cite{Falk1986,Elden2005}, where a Carleman weight is appropriately applied to ''convexify'' the energy functional logarithmically. Lastly, we wish to mention that the Cauchy problem posed in unbounded domains has also been considered in, e.g., \cite{Karimi2017,Qiu2008}.

Different from the above-mentioned regularization methods, we would like to study in this work a modified quasi-reversibility (QR) method, which is a PDE-based approach. The QR method was originally mentioned in the monograph by Latt\`es and Lions; see \cite{LL67} by perturbing the unbounded operator - the main cause of the instability. The modified QR method under investigation has been commenced in the pioneering work \cite{Nguyen2019}, where the authors established two operators along with their conditional estimates to guarantee the strong convergence of the scheme solving quasi-linear parabolic equations backwards in time. The key ingredient of the method is that using a suitable perturbation, the ill-posed problem turns to be a forward-like problem in which we can prove its conditional well-posedness. It is, on the other hand, certain that numerical solutions for forward problems are well-studied nowadays. Recently, this modified QR method has been applied to the Cauchy problem for the Laplace system in \cite{Khoa2020a}. In this regard, we regularize the Cauchy-Laplace problem by the corresponding initial-value hyperbolic problem. Using the same idea, in the present work, we verify the applicability of this method to solve the Cauchy problem for the Helmholtz equation. By the possible involvement of large frequencies $k$, we, however, remark that the perturbation should be chosen appropriately. Accordingly, we focus ourselves on specific perturbation and stabilized operators. 

\subsection{Organization of the paper}

The paper is organized as follows. In section \ref{sec:2}, we recall some preliminaries concerning the ill-posedness of the Cauchy problem and how we derive the modified QR scheme from the original PDE. We also specify the perturbation operator and the corresponding stabilized operator in our regularized problem. Conditional estimates of these operators are deduced accordingly. Then, we analyze the conditional well-posedness of the regularized problem and the strong convergence of the scheme in section \ref{sec:3}. The Lipschitz stability of the scheme also follows. In section \ref{sec:4}, we investigate the corresponding iterative scheme. Finally, some numerical examples are provided in section \ref{sec:num} to corroborate our theoretical analysis.

\section{Preliminaries}\label{sec:2}

Let $A$ be either a Banach space or a Hilbert space. We call $A'$ the dual space of $A$. For a certain Banach space $A$, $\left\Vert \cdot\right\Vert_{A}$ stands for the $A$-norm. When $A$ is Hilbert, we define the $A$-norm of $u$ as $\left\Vert u\right\Vert_{A}^{2} = \left\langle u,u\right\rangle_{A}$, where $\left\langle \cdot,\cdot\right\rangle_{A}$ is the corresponding inner product. Throughout the paper, we will use $\left\langle \cdot,\cdot\right\rangle $ to indicate either the scalar product in $L^2(0,1)$ or the dual pairing of a continuous linear functional and
an element of a function space. We thereby denote by $\left\Vert \cdot\right\Vert$ the $L^2(0,1)$-norm. 

The Cauchy problem of Helmholtz equation is well known to be unstable with respect to any small perturbation of the data. Based on the zero Dirichlet boundary condition \eqref{boundary},
the Laplace operator $-\partial^2/\partial y^2$ is non-negative. According to the
standard result for the Dirichlet eigenvalue problem, there exists an
orthonormal basis $\left\{ \phi_{j}\right\} $ of $L^{2}\left(0,1\right)$
such that $\phi_{j}\in H^{1}\left(0,1\right)\cap C^{\infty}\left[0,1\right]$
and $-d^2/d y^2\phi_{j}\left(y\right)=\mu_{j}\phi_{j}\left(y\right)$.
The Dirichlet eigenvalues $\mu_{j}$ in this case form an infinite sequence such
that $
0\le \mu_{0}<\mu_{1}<\mu_{2}<\ldots,\text{and }\lim_{j\to\infty}\mu_{j}=\infty
$. It follows from \eqref{helmholtz} and \eqref{boundary1} that we obtain the following initial-value differential system:
\begin{align}\label{ODE}
	\begin{cases}
		\frac{d^{2}}{dx^{2}}\left\langle u\left(x,\cdot\right),\phi_{j}\right\rangle -\lambda_{j,k}\left\langle u\left(x,\cdot\right),\phi_{j}\right\rangle =0,\\
		\left\langle u\left(0,\cdot\right),\phi_{j}\right\rangle =\left\langle u_{0},\phi_{j}\right\rangle ,\quad 
		\frac{d}{dx}\left\langle u\left(0\right),\phi_{j}\right\rangle =0 .
	\end{cases}
\end{align}

In \eqref{ODE}, $\lambda_{j,k} = \mu_j - k^2$. By this way, we solve system \eqref{ODE} in each of the following set of Fourier frequencies:
\[
A_{1}:=\left\{j\in\mathbb{N}:\lambda_{j,k}>0\right\} ,\quad A_{2}:=\left\{ j\in\mathbb{N}:\lambda_{j,k}=0\right\} ,\quad A_{3}:=\left\{ j\in\mathbb{N}:\lambda_{j,k}<0\right\} .
\]
It is also straightforward to see that $\phi_{j}(y)=\sqrt{2}\sin\left(j\pi y\right),\; \mu_{j}=j^2\pi^2$. In addition,  $\left\{\phi_{j}'/\sqrt{\mu_j}\right\}_{j\in\mathbb{N}^{*}}$ is an orthonormal basis of $L^2(0,1)$. Therefore, it holds that
\begin{align}\label{normuy}
\left\Vert u_{y}\right\Vert ^{2}=\sum_{j\in\mathbb{N}^{*}}\left|\left\langle u_{y},\frac{\phi_{j}'}{\sqrt{\mu_{j}}}\right\rangle \right|^{2}=\sum_{j\in\mathbb{N}^{*}}\left|\left\langle u,\frac{\phi_{j}''}{\sqrt{\mu_{j}}}\right\rangle \right|^{2}=\sum_{j\in\mathbb{N}^{*}}\mu_{j}\left|\left\langle u,\phi_{j}\right\rangle \right|^{2}.
\end{align}

\begin{theorem}\label{thm:instability}
	The Fourier coefficient $\left\langle u\left(x,\cdot\right),\phi_{j}\right\rangle$ has the form:
	\begin{align}\label{uuu}
	\left\langle u\left(x,\cdot\right),\phi_{j}\right\rangle =\begin{cases}
		\cosh\left(\sqrt{\lambda_{j,k}}x\right)\left\langle u_{0},\phi_{j}\right\rangle  & \text{in }A_{1},\\
		\cos\left(\sqrt{-\lambda_{j,k}}x\right)\left\langle u_{0},\phi_{j}\right\rangle  & \text{in }A_{3}.
	\end{cases}
	\end{align}
\end{theorem}
\begin{proof}
	Proof of the theorem can be proceeded as in \cite{Tuan2017}. In $A_1$, solving system \eqref{ODE} gives
	\begin{align}\label{uu}
		\left\langle u\left(x,\cdot\right),\phi_{j}\right\rangle =C_{1}e^{x\sqrt{\lambda_{j,k}}}+C_{2}e^{-x\sqrt{\lambda_{j,k}}}.
	\end{align}
	Therefore, when $x=0$, it yields
	\begin{align}\label{aa}
		\left\langle u_{0},\phi_{j}\right\rangle =\left\langle u\left(0,\cdot\right),\phi_{j}\right\rangle =C_{1}+C_{2}.
	\end{align}
	On the other hand, we compute that
	\[
	\frac{d}{dx}\left\langle u\left(x,\cdot\right),\phi_{j}\right\rangle =\sqrt{\lambda_{j,k}}\left(C_{1}e^{x\sqrt{\lambda_{j,k}}}-C_{2}e^{-x\sqrt{\lambda_{j,k}}}\right).
	\]
	When $x=0$, we arrive at
	\begin{align}\label{bb}
		0 =\frac{d}{dx}\left\langle u\left(0,\cdot\right),\phi_{j}\right\rangle =\sqrt{\lambda_{j,k}}\left(C_{1}-C_{2}\right)
	\end{align}
	Combining \eqref{aa} and \eqref{bb}, we have
	$
	C_{1}=C_{2}=\frac{1}{2}\left\langle u_{0},\phi_{j}\right\rangle.
	$
	By back-substitution of these $C_1$ and $C_2$ into  \eqref{uu}, the Fourier coefficient of $u$ is formulated by
	\begin{align}\label{uuu}
		\left\langle u\left(x,\cdot\right),\phi_{j}\right\rangle =\cosh\left(\sqrt{\lambda_{j,k}}x\right)\left\langle u_{0},\phi_{j}\right\rangle .
	\end{align}
	In $A_3$, we do the same way and obtain
	$
	\left\langle u\left(x,\cdot\right),\phi_{j}\right\rangle =\cos\left(\sqrt{-\lambda_{j,k}}x\right)\left\langle u_{0},\phi_{j}\right\rangle
	$. This completes the proof of the theorem.
\end{proof}

Now, we show a very important relation of these Fourier frequencies in the following theorem.

\begin{theorem} \label{thm:relation}
	Taking into account set $A_1$, the Fourier coefficient of $u$ satisfies the following relation:
	\begin{align}\label{relation}
		\lambda_{j,k}e^{\left(1-x\right)\sqrt{\lambda_{j,k}}}\left(\left\langle u\left(x,\cdot\right),\phi_{j}\right\rangle +\frac{1}{\sqrt{\lambda_{j,k}}}\left\langle u_{x}\left(x,\cdot\right),\phi_{j}\right\rangle \right)=\lambda_{j,k}\left\langle u\left(1,\cdot\right),\phi_{j}\right\rangle +\sqrt{\lambda_{j,k}}\left\langle u_{x}\left(1,\cdot\right),\phi_{j}\right\rangle.
	\end{align}
\end{theorem}
\begin{proof}
	From \eqref{uuu}, we can compute that
	\begin{align}\label{ux}
		\frac{1}{\sqrt{\lambda_{j,k}}}\left\langle u_{x}\left(x,\cdot\right),\phi_{j}\right\rangle =\sinh\left(\sqrt{\lambda_{j,k}}x\right)\left\langle u_{0},\phi_{j}\right\rangle .
	\end{align}
	Thus, we take $x=1$ in \eqref{ux} and in \eqref{uuu} and then combine the resulting formulations to obtain \eqref{relation}. We complete the proof of the theorem. 
\end{proof}

Practically, the data $u_0$ in \eqref{boundary1} always contain noise of measurement. Therefore, we assume to have $u_0^{\varepsilon}\in H^1(0,1)$ as the noisy data such that for $\varepsilon \in (0,1)$,
\begin{align}\label{noisecondition}
	\left\Vert u_{0}^{\varepsilon}-u_{0}\right\Vert _{H^1\left(0,1\right)}\le\varepsilon.
\end{align}
By Theorem \ref{thm:instability}, our Cauchy problem is exponentially unstable in $A_1$ due to the natural growth of the hyperbolic cosine function. Any small perturbation of the initial data $u_0$ may cause a huge error when computing solution $u$ of the Cauchy problem. In this work, we then adapt our recent modified quasi-reversibility method (cf. \cite{Khoa2020a} for elliptic operators and \cite{Nguyen2019} for parabolic operators) to solve our system \eqref{helmholtz}--\eqref{boundary1}. To do so, we rewrite (\ref{helmholtz}) as
\begin{equation}
	u_{xx}-u_{yy}+2u_{yy} + 2k^2 u= k^2 u.\label{eq:4}
\end{equation}
We then perturb (\ref{eq:4}) by a linear mapping $\mathbf{Q}$
and take $\mathbf{P}=\mathbf{Q}+2\partial^2/\partial y^2 + 2k^2$.
Henceforth, we arrive at
\begin{align}\label{regu1}
	u_{xx}-u_{yy}+\mathbf{P}u = k^2 u\quad \text{in }\Omega.
\end{align}
It is worth mentioning that together with the boundary condition \eqref{boundary} and the Cauchy data \eqref{boundary1} with measurement  $u_0^{\varepsilon}$,  \eqref{regu1} forms a system of linear wave equation. Herewith, $x$ becomes a parametric time variable. Since the noise level $\varepsilon$ is involved, we then seek a sequence of $\left\{u^{\varepsilon}\right\}_{\varepsilon>0}$ satisfying the following system:
\begin{align}\label{reguHelm}
	\begin{cases}
		u_{xx}^{\varepsilon}-u_{yy}^{\varepsilon}+\mathbf{P}u^{\varepsilon} =k^2 u^{\varepsilon} & \text{in }\Omega,\\
		u^{\varepsilon}\left(x,0\right)=u^{\varepsilon}\left(x,1\right)=0 & \text{for }x\in\left(0,1\right),\\
		u^{\varepsilon}\left(0,y\right)=u_{0}^{\varepsilon}\left(y\right),\quad u_{x}^{\varepsilon}\left(0,y\right)=0 & \text{for }y\in\left(0,1\right).
	\end{cases}
\end{align}
Cf. \cite{Khoa2020a,Nguyen2019}, $\mathbf{Q}$ is called \textit{perturbation} as it is to ``absorb'' high Fourier frequencies in the Laplace operator, and $\mathbf{P}$ is called \textit{stabilized operator} as it only contains large enough Fourier frequencies serving for the convergence of the scheme. Let $\gamma>1$. Consider $B:=\left\{ j\in A_1:\lambda_{j,k}>\log^2(\gamma) \right\}$. We choose the following truncation operator:
\begin{align}\label{Q}
	\mathbf{Q}u(x,\cdot)=2\sum_{j\in  B}\lambda_{j,k}\left\langle u(x,\cdot),\phi_{j}\right\rangle \phi_{j}  +2\sum_{j\in A_3}\lambda_{j,k}\left\langle u\left(x,\cdot\right),\phi_{j}\right\rangle \phi_{j}:= \mathbf{Q_1}u(x,\cdot) + \mathbf{Q_2}u(x,\cdot).
\end{align}
As to the corresponding stabilized operator $\mathbf{P}$, we find that
\begin{align*}
	\mathbf{P}u(x,\cdot) & =2\sum_{j\in B\cup A_3}\lambda_{j,k}\left\langle u(x,\cdot),\phi_{j}\right\rangle \phi_{j}-2\sum_{j\in\mathbb{N}}\mu_{j}\left\langle u(x,\cdot),\phi_{j}\right\rangle \phi_{j} + 2k^{2}\sum_{j\in\mathbb{N}}\left\langle u(x,\cdot),\phi_{j}\right\rangle \phi_{j}\\
	& =2\sum_{j\in B\cup A_3}\left(\lambda_{j,k}-\mu_{j}\right)\left\langle u(x,\cdot),\phi_{j}\right\rangle \phi_{j}-2\sum_{j\in\mathbb{N}\backslash (B\cup A_3)}\mu_{j}\left\langle  u(x,\cdot),\phi_{j}\right\rangle \phi_{j} + 2k^{2}\sum_{j\in\mathbb{N}}\left\langle u(x,\cdot),\phi_{j}\right\rangle \phi_{j}\\
	& =-2\sum_{j\in\mathbb{N}\backslash (B\cup A_3)}\lambda_{j,k}\left\langle u(x,\cdot),\phi_{j}\right\rangle \phi_{j}.
\end{align*}

In view of Parseval's identity, we now estimate that
\begin{align}
	\left\Vert \mathbf{Q}_1u(x,\cdot)\right\Vert ^{2}  =4\sum_{j\in B}e^{-2\sqrt{\lambda_{j,k}}}\lambda_{j,k}^{2}e^{2\sqrt{\lambda_{j,k}}}\left|\left\langle u(x,\cdot),\phi_{j}\right\rangle \right|^{2} \le4\gamma^{-2}\sum_{j\in B}\lambda_{j,k}^{2}e^{2\sqrt{\lambda_{j,k}}}\left|\left\langle u(x,\cdot),\phi_{j}\right\rangle \right|^{2}.\label{estQ}
\end{align}
By using \eqref{relation} obtained in Theorem \ref{thm:relation}, we have
\begin{align*}
	&\sup_{x\in\left[0,1\right]}\sum_{j\in B}\lambda_{j,k}^{2}e^{2\sqrt{\lambda_{j,k}}}\left|\left\langle u(x,\cdot),\phi_{j}\right\rangle \right|^{2}\\
	& \le\sup_{x\in\left[0,1\right]}\left[\sum_{j\in B}\lambda_{j,k}^{2}e^{2\left(1-x\right)\sqrt{\lambda_{j,k}}}\left(\left\langle u\left(x,\cdot\right),\phi_{j}\right\rangle +\frac{1}{\sqrt{\lambda_{j,k}}}\left\langle u_{x}\left(x,\cdot\right),\phi_{j}\right\rangle \right)^{2}\right]\\
	& \le\sum_{j\in B}\left(\lambda_{j,k}\left\langle u\left(1,\cdot\right),\phi_{j}\right\rangle +\sqrt{\lambda_{j,k}}\left\langle u_{x}\left(1,\cdot\right),\phi_{j}\right\rangle \right)^{2} \le 2\left\Vert u\left(1,\cdot\right)\right\Vert _{H^{2}\left(0,1\right)}^{2}+2\left\Vert u_{x}\left(1,\cdot\right)\right\Vert _{H^{1}\left(0,1\right)}^{2}.
\end{align*}
by means of $\left\langle u\left(x,\cdot\right),\phi_{j}\right\rangle \left\langle u_{x}\left(x,\cdot\right),\phi_{j}\right\rangle  \ge 0$; cf. \eqref{uuu} and \eqref{ux}. Now we estimate $\mathbf{Q}_2u$ as follows. Observe that if $\log\left(\gamma\right)\ge k$, then $\mu_{i}-k^{2}\ge k^{2}-\mu_{j}>0$ for $i\in B$ and $j\in A_{3}$. This means that $\lambda_{i,k}\ge\left|\lambda_{j,k}\right|$ for $i\in B$ and $j\in A_{3}$. Therefore, we estimate that
\begin{align}\label{Q2}
\left\Vert \mathbf{Q}_2u(x,\cdot)\right\Vert ^{2} = 4\sum_{j\in A_3}\left|\lambda_{j,k}\right|^{2}\left|\left\langle u\left(x,\cdot\right),\phi_{j}\right\rangle \right|^{2}
\le 4\sum_{j\in B}\left|\lambda_{j,k}\right|^{2}\left|\left\langle u\left(x,\cdot\right),\phi_{j}\right\rangle \right|^{2} = \left\Vert \mathbf{Q}_1u(x,\cdot)\right\Vert ^{2}.
\end{align}
Henceforth, we can assume that the true solution satisfies $u(1,\cdot)\in H^2(0,1)$ and $u_x(1,\cdot)\in H^1(0,1)$ to gain the strong convergence of the scheme. Note now that $\mathbf{P}$ is computable, which is relevant to our numerical simulation, compared to many other modified kernel regularization methods. Moreover, since in $\mathbb{N}\backslash (B\cup A_3)$ it holds that $0\le \lambda_{j,k}\le \log^2(\gamma)$, we, according to Parseval's identity and using \eqref{normuy}, get that
\begin{align}\label{estP}
	\left\Vert \mathbf{P}u(x,\cdot)\right\Vert ^{2}\le 4\log^2(\gamma)\left\Vert u_y(x,\cdot)\right\Vert ^{2}.
\end{align}
\begin{remark}
	In section \ref{sec:3} below, we will prove that the approximate solution $u^{\varepsilon}$ approaches $u$ under an appropriate choice of $\gamma$ dependent of the noise level $\varepsilon$. Complying with that, we below denote our operators by $\mathbf{Q}_{\varepsilon}$ and $\mathbf{P}_{\varepsilon}$ in lieu of, as above, $\mathbf{Q}$ and $\mathbf{P}$, respectively.
\end{remark}

\section{Analysis of the regularization scheme}\label{sec:3}

We now formulate theorems for the weak solvability of system \eqref{reguHelm} and convergence analysis of the corresponding regularization scheme. When doing so, we provide the definition of weak solution as follows.

\begin{definition}[Weak solution]\label{def:weak1}
	For each $\varepsilon>0$, a function $u^{\varepsilon}:[0,1]\to H^1_0(0,1)$ is said to be a weak solution to system \eqref{reguHelm} if
	\begin{itemize}
		\item $u^{\varepsilon}\in C([0,1];H^1_0(0,1)), \partial_{x}u^{\varepsilon}\in C([0,1];L^2(0,1)), \partial^2_{x^2}u^{\varepsilon}\in L^2(0,1;(H^1(0,1))')$;
		\item For every test function $\psi\in H_0^1(0,1)$, it holds that
		\begin{align}
			\left\langle \frac{\partial^{2}u^{\varepsilon}}{\partial x^{2}},\psi\right\rangle +\left\langle \frac{\partial u^{\varepsilon}}{\partial y},\frac{\partial\psi}{\partial y}\right\rangle +\left\langle \mathbf{P}_{\varepsilon}u^{\varepsilon},\psi\right\rangle = k^{2}\left\langle u^{\varepsilon},\psi\right\rangle \quad \text{for a.e. in }(0,1);
		\end{align}
		\item $u^{\varepsilon}(0)=u_0^{\varepsilon}\in H^1(0,1), \partial_{x}u^{\varepsilon}(0)=0$.
	\end{itemize}
\end{definition}

\begin{theorem}[Existence and uniqueness of a weak regularized solution]\label{thm:weak}
	For each $\varepsilon>0$, system \eqref{reguHelm} admits a unique weak solution in the sense of Definition \ref{def:weak1}. Moreover, it holds that $u^{\varepsilon}\in C([0,1];H_0^1(0,1))$ and $\partial_{x}u^{\varepsilon} \in C([0,1];L^2(0,1))$.
\end{theorem}

\begin{proof}
	To prove this theorem, we employ the standard Galerkin approximation. Consider the $n$-dimensional of $H^1_0(0,1)$ generated by $\phi_0,\phi_1,\ldots,\phi_n$. For each $n\in\mathbb{N}$, we take into account the following Galerkin projection for approximation of \eqref{reguHelm}:
	\begin{align}\label{Galerkinprojec}
		u_{n}^{\varepsilon}(x,y) = \sum_{j=0}^nU_{jn}^{\varepsilon}(x)\phi_j(y).
	\end{align}
	This function $u_{n}^{\varepsilon}$ is hereby the solution of the following approximate equation:
	\begin{align}\label{Galerkinequa}
		\left\langle \frac{\partial^{2}u_{n}^{\varepsilon}}{\partial x^{2}},\psi\right\rangle +\left\langle \frac{\partial u_{n}^{\varepsilon}}{\partial y},\frac{\partial\psi}{\partial y}\right\rangle +\left\langle \mathbf{P}_{\varepsilon}u_{n}^{\varepsilon},\psi\right\rangle = k^{2}\left\langle u_{n}^{\varepsilon},\psi\right\rangle \quad \text{for } \psi\in \mathbb{S}_n\text{ and a.e. in }(0,1).
	\end{align}
	This Galerkin equation is endowed with the initial data $\partial_xu_{n}^{\varepsilon}(0,y) = 0$ and
	\begin{align}\label{stronginH1}
		u_{n}^{\varepsilon}(0,y) = \sum_{j=0}^n\left(U_0^{\varepsilon}\right)_{jn}\phi_j(y) \xrightarrow{\text{strongly in $H^1(0,1)$}} u_0^{\varepsilon} \text{ as } n\to\infty.
	\end{align}
	Now, let $\psi = \phi_j$, where recall that $\left\{\phi_{j}\right\}$ is the orthonormal basis of $L^2(0,1)$. Then functions $U_{jn}^{\varepsilon}$ are solutions to the Cauchy problem for the system of $n$ vectorial ordinary
	differential equations:
	\begin{align}\label{vecsystem}
		\begin{cases}
			\displaystyle{\dfrac{d^2}{dx^2}U_{jn}^{\varepsilon}+(\mu_j-k^2)U_{jn}^{\varepsilon}+\sum_{i=0}^nU_{in}^{\varepsilon}\left\langle\textbf{P}^{\varepsilon}\phi_i,\phi_j\right\rangle = 0},\\
			U_{jn}^{\varepsilon}(0) = \left(U_0^{\varepsilon}\right)_{jn},\dfrac{d}{dx}U_{jn}^{\varepsilon}(0) = 0.
		\end{cases}
	\end{align}
	For any $n\in\mathbb{N}$, we put $Z_{jn}^{\varepsilon}=d U_{jn}^{\varepsilon}/dx$. It then deduces from \eqref{vecsystem} that
	\begin{align*}
		\dfrac{d}{dx}\begin{bmatrix}
			U_{jn}^{\varepsilon}\\
			Z_{jn}^{\varepsilon}
		\end{bmatrix} = \begin{bmatrix}
			0 & 1\\
			k^2-\mu_j & 0
		\end{bmatrix}\begin{bmatrix}
			U_{jn}^{\varepsilon}\\
			Z_{jn}^{\varepsilon}
		\end{bmatrix}+\begin{bmatrix}
			0\\
			-\sum_{i=0}^nU_{in}^{\varepsilon}\left\langle\textbf{P}_{\varepsilon}\phi_i,\phi_j\right\rangle 
		\end{bmatrix},\quad\begin{bmatrix}
			U_{jn}^{\varepsilon}(0)\\
			Z_{jn}^{\varepsilon}(0)
		\end{bmatrix} = \begin{bmatrix}
			\left(U_{jn}^{\varepsilon}\right)_{jn}\\
			0
		\end{bmatrix}.
	\end{align*}
	Let $z_{jn}^{\varepsilon} = [U_{jn}^{\varepsilon}, Z_{jn}^{\varepsilon}]^T$. Then solving the above closed-form initial-value differential problem, we obtain the following integral equation:
	\begin{align}\label{equaintervec}
		z_{jn}^{\varepsilon}(x) = z_{jn}^{\varepsilon}(0)+A_j^k\int_0^xz_{jn}^{\varepsilon}(s)ds+\int_0^xF_j(z^{\varepsilon})(s)ds.
	\end{align}
	In \eqref{equaintervec}, we denote by
	\begin{align*}
		A_j^k = \begin{bmatrix}
			0 & 1\\
			k^2-\mu_j & 0
		\end{bmatrix}, \quad F_j(z^{\varepsilon}) = \begin{bmatrix}
			0\\
			-\sum_{i=0}^nU_{in}^{\varepsilon}\left\langle\textbf{P}_{\varepsilon}\phi_i,\phi_j\right\rangle 
		\end{bmatrix}.
	\end{align*}
	We now define $z^{\varepsilon} = \left[z_{0n}^{\varepsilon},z_{1n}^{\varepsilon},\ldots,z_{nn}^{\varepsilon}\right]\in\mathbb{R}^{2(n+1)}$ and denote $H_j[z^{\varepsilon}]$ by the right-hand side of \eqref{equaintervec}. This results in the fixed-point form $z^{\varepsilon}(x) = H[z^{\varepsilon}](x)$ where $H[z^{\varepsilon}] = \left[H_0[z^{\varepsilon}],H_1[z^{\varepsilon}],\ldots,H_n[z^{\varepsilon}]\right]$. Define the norm of $Y = C\left([0,1];\mathbb{R}^{2(n+1)}\right)$ as
	\begin{align}\label{definednorm}
		\rVert{c}\rVert_Y = \sup_{x\in[0,1]}\sum_{j=0}^n|c_j(x)|,\quad c(x) = [c_0(x),c_1(x),\ldots,c_n(x)]\in \mathbb{R}^{2(n+1)}.
	\end{align}
	We claim that there exists $m_0\in\mathbb{N}^{*}$ such that the operator $H^{m_0}:=H[H^{m_0-1}]:Y\to Y$ is a contraction mapping. Indeed, by induction we can prove that 
	\begin{align}\label{induction}
		\left|H_j^m[z_1^{\varepsilon}](x)-H_j^m[z_2^{\varepsilon}](x)\right|\leq \left[\sqrt{1+(k^2-\mu_j)^2}+2\log\left(\gamma\right)\right]^{m}\dfrac{x^m}{m!}\left\rVert{z_1^{\varepsilon}-z_2^{\varepsilon}}\right\rVert_Y
	\end{align}
	for $m\in\mathbb{N}^{*}$ and for any $z_1^{\varepsilon},z_2^{\varepsilon}\in Y$. Observe that the inductive hypothesis is true when $m = 1$. In particular, in view of the fact that
	\[
	\left\langle \mathbf{P}_{\varepsilon}\phi_{i},\phi_{j}\right\rangle =\begin{cases}
		-2\lambda_{i,k} & \text{if }i=j\in\mathbb{N}\backslash\left(B\cup A_{3}\right),\\
		0 & \text{elsewhere},
	\end{cases}
	\]
	we can estimate that
	\begin{align}
		\left|H_j[z_1^{\varepsilon}](x)-H_j[z_2^{\varepsilon}](x)\right|&\leq\int_0^x\bigg(\sqrt{1+(k^2-\mu_j)^2}\big|z_{1j}^{\varepsilon}(s)-z_{2j}^{\varepsilon}(s)\big|+\sum_{i=0}^n\big|U_{1i}^{\varepsilon}-U_{2i}^{\varepsilon}\big|\big|\left\langle\textbf{P}_{\varepsilon}\phi_i,\phi_j\right\rangle\big|\bigg)ds \nonumber\\
		&\leq\int_0^x\bigg(\sqrt{1+(k^2-\mu_j)^2}\big|z_{1j}^{\varepsilon}(s)-z_{2j}^{\varepsilon}(s)\big|+\sum_{i=0}^n 2\left|\lambda_{i,k}\right|\big|U_{1j}^{\varepsilon}-U_{2j}^{\varepsilon}\big|\bigg)ds \nonumber \\
		&\leq \bigg[\sqrt{1+(k^2-\mu_j)^2}+2\log\left(\gamma\right)\bigg]x\rVert{z_1^{\varepsilon}-z_2^{\varepsilon}}\rVert_Y. \label{m=1}
	\end{align}
	For $m= m_0 > 1$, we assume that
	\begin{align*}
		\left|H_j^{m_0}[z_1^{\varepsilon}](x)-H_j^{m_0}[z_2^{\varepsilon}](x)\right|\leq \bigg[\sqrt{1+(k^2-\mu_j)^2}+2\log\left(\gamma\right)\bigg]^{m}\dfrac{x^{m_0}}{m_0!}\rVert{z_1^{\varepsilon}-z_2^{\varepsilon}}\rVert_Y.
	\end{align*}
	We then want to prove that \eqref{induction} also holds true for $m=m_0 + 1$. Using the same token as in \eqref{m=1}, we estimate that
	\begin{align*}
		\left|H_j^{m_0+1}[z_1^{\varepsilon}](x)-H_j^{m_0+1}[z_2^{\varepsilon}](x)\right|&\leq\int_0^x\bigg[\sqrt{1+(k^2-\mu_j)^2}+2\log\left(\gamma\right)\bigg]\big|H_j^{m_0}[z_1^{\varepsilon}](s)-H_j^{m_0}[z_2^{\varepsilon}](s)\big|ds\\
		&\leq\int_0^x\bigg[\sqrt{1+(k^2-\mu_j)^2}+2\log\left(\gamma\right)\bigg]^{m_0+1}\dfrac{s^{m_0}}{m_0!}\rVert{z_1^{\varepsilon}-z_2^{\varepsilon}}\rVert_Yds\\
		&= \bigg[\sqrt{1+(k^2-\mu_j)^2}+2\log\left(\gamma\right)\bigg]^{m_0+1}\dfrac{x^{m_0+1}}{(m_0+1)!}\rVert{z_1^{\varepsilon}-z_2^{\varepsilon}}\rVert_Y
	\end{align*}
	Henceforth, our inductive hypothesis \eqref{induction} is true for any $m\in\mathbb{N}$. It also leads to the following estimate in the norm of $Y$:
	\begin{align*}
		\big\rVert{H^m[z_1^{\varepsilon}]-H^m[z_2^{\varepsilon}]}\big\rVert_Y\leq \sum_{j=0}^n\bigg[\sqrt{1+(k^2-\mu_j)^2}+2\log\left(\gamma\right)\bigg]^{m}\dfrac{x^{m}}{m!}\rVert{z_1^{\varepsilon}-z_2^{\varepsilon}}\rVert_Y.
	\end{align*}
	Since the following limit holds true
	\begin{align*}
		\lim_{m\to\infty}\sum_{j=0}^n\bigg[\sqrt{1+(k^2-\mu_j)^2}+2\log\left(\gamma\right)\bigg]^{m}\dfrac{x^{m}}{m!} = 0
	\end{align*}
	we then can find $m_0\in\mathbb{N}$ sufficiently large such that
	\begin{align*}
		\sum_{j=0}^n\bigg[\sqrt{1+(k^2-\mu_j)^2}+2\log\left(\gamma\right)\bigg]^{m_0}\dfrac{x^{m_0}}{m_0!}<1.
	\end{align*}
	This clearly indicates the existence of a constant $K\in[0,1)$ satisfies $\big\rVert{H^{m_0}[z_1^{\varepsilon}]-H^{m_0}[z_2^{\varepsilon}]}\big\rVert_Y\leq K\rVert{z_1^{\varepsilon}-z_2^{\varepsilon}}\rVert_Y$. In other words, $H^{m_0}$ is a contraction mapping from $Y$ onto itself. By the Banach fixed-point theorem, there exists a unique $z^{\varepsilon}\in Y$ such that $H^{m_0}[z^{\varepsilon}] = z^{\varepsilon}$. As $H^{m_0}\left[H[z^{\varepsilon}]\right] = H\left[H^{m_0}[z^{\varepsilon}]\right] = H[z^{\varepsilon}]$, the integral equation $z^{\varepsilon} = H[z^{\varepsilon}]$ admits a unique solution in $Y$. Hence, this results in the existence and uniqueness of $U^{\varepsilon}_{jn}\in C^1([0,1])$ solutions to system \eqref{vecsystem} for any fixed $n\in\mathbb{N}$.
	
	By $U^{\varepsilon}_{jn}\in C^1([0,1])$, we have $\partial_x u_{n}^{\varepsilon}\in C([0,1];\mathbb{S}_{n})$. Multiply both sides of \eqref{Galerkinequa} by $e^{-rx}$ and then put $v^{\varepsilon}_n = e^{-rx} u_{n}^{\varepsilon}$. Therefore, we obtain the Galerkin equation for $v^{\varepsilon}_{n}$ as follows:
	\begin{align}\label{Galerkinequav}
		\left\langle \frac{\partial^{2}v_{n}^{\varepsilon}}{\partial x^{2}},\psi\right\rangle +\left\langle \frac{\partial v_{n}^{\varepsilon}}{\partial y},\frac{\partial\psi}{\partial y}\right\rangle + 2r\left\langle \frac{\partial v_n^{\varepsilon}}{\partial x},\psi  \right\rangle + \left\langle \mathbf{P}_{\varepsilon}v_{n}^{\varepsilon},\psi\right\rangle = (k^{2} - r^{2})\left\langle v_{n}^{\varepsilon},\psi\right\rangle \quad \text{for } \psi\in \mathbb{S}_n\text{ and a.e. in }(0,1).
	\end{align}
	Thus, we choose $\psi = \partial_x v_{n}^{\varepsilon}$ in \eqref{Galerkinequav} and $r>k$ to get that
	\begin{align}\label{estim1}
		\begin{split}
			& \dfrac{d}{dx}\bigg[\big\rVert{\partial_x v_{n}^{\varepsilon}(x,\cdot)}\big\rVert^2+\big\rVert{\partial_y v_{n}^{\varepsilon}(x,\cdot)}\big\rVert^2 + (r^2 -k^2)\big\rVert{v_{n}^{\varepsilon}(x,\cdot)}\big\rVert^2\bigg] \\
			& = -2\left\langle\textbf{P}_{\varepsilon}v_{n}^{\varepsilon}(x,\cdot),\partial_xv_{n}^{\varepsilon}(x,\cdot)\right\rangle 
			- 4r \big\rVert{\partial_x v_{n}^{\varepsilon}(x,\cdot)}\big\rVert^2 
			\leq 2\log\left(\gamma\right)\big\rVert{\partial_yv_{n}^{\varepsilon}(x,\cdot)}\big\rVert^2+2\log\left(\gamma\right)\big\rVert{\partial_xv_{n}^{\varepsilon}(x,\cdot)}\big\rVert^2.
		\end{split}
	\end{align}
	By integrating the estimate \eqref{estim1} with respect to $x$, we arrive at
	\begin{align*}
		\big\rVert{\partial_xv_{n}^{\varepsilon}(x,\cdot)}\big\rVert^2+\big\rVert{\partial_yv_{n}^{\varepsilon}(x,\cdot)}\big\rVert^2 + (r^2 -k^2)\big\rVert{v_{n}^{\varepsilon}(x,\cdot)}\big\rVert^2 & \leq \big\rVert{\partial_x v_{n}^{\varepsilon}(0,\cdot)}\big\rVert^2+ \big\rVert{\partial_yv_{n}^{\varepsilon}(0,\cdot)}\big\rVert^2+(r^2 -k^2)\big\rVert{v_{n}^{\varepsilon}(0,\cdot)}\big\rVert^2 \\
		&+2\log\left(\gamma\right)\int_0^x\left(\big\rVert{\partial_yv_{n}^{\varepsilon}(s,\cdot)}\big\rVert^2+\big\rVert{\partial_xv_{n}^{\varepsilon}(s,\cdot)}\big\rVert^2\right)ds.
	\end{align*}
	By using Gronwall’s inequality, we thus get
	\begin{align}\label{Gronwallesti}
		\big\rVert{\partial_xv_{n}^{\varepsilon}(x,\cdot)}\big\rVert^2+\big\rVert{\partial_yv_{n}^{\varepsilon}(x,\cdot)}\big\rVert^2
		& + (r^2 -k^2)\big\rVert{v_{n}^{\varepsilon}(x,\cdot)}\big\rVert^2 \nonumber \\ & 
		\leq \left(\big\rVert{\partial_x v_{n}^{\varepsilon}(0,\cdot)}\big\rVert^2+ \big\rVert{\partial_y v_{n}^{\varepsilon}(0,\cdot)}\big\rVert^2+(r^2 -k^2)\big\rVert{v_{n}^{\varepsilon}(0,\cdot)}\big\rVert^2\right)\gamma^{2x}.
	\end{align}
	Since $v^{\varepsilon}_{n}(0,\cdot) = u^{\varepsilon}_{n}(0,\cdot)$ and $\partial_x v^{\varepsilon}_{n}(0,\cdot) = -rv_{n}^{\varepsilon}(0,\cdot) + \partial_x u^{\varepsilon}_{n}(0,\cdot)$, there exists a constant $C>0$ independent of $n$ such that $\big\rVert{\partial_x v_{n}^{\varepsilon}(0,\cdot)}\big\rVert^2+ \big\rVert{\partial_y v_{n}^{\varepsilon}(0,\cdot)}\big\rVert^2+(r^2 -k^2)\big\rVert{v_{n}^{\varepsilon}(0,\cdot)}\big\rVert^2 \le C$; cf. \eqref{stronginH1} and \eqref{noisecondition}. Therefore, for any $n\in\mathbb{N}$ we obtain 
	\begin{align*}
		&v_{n}^{\varepsilon} \text{ is bounded in } L^{\infty}\left(0,1;H^1(0,1)\right),\\
		&\partial_x v_{n}^{\varepsilon} \text{ is bounded in } L^{\infty}\left(0,1;L^2(0,1)\right).
	\end{align*}
	By the Banach–Alaoglu theorem, we can extract a subsequence of  $v_{n}^{\varepsilon}$ (which we still denote by $\big\{v_{n}^{\varepsilon}\big\}_{n\in\mathbb{N}}$) such that for each $\varepsilon > 0$,
	\begin{align*}
		&v_{n}^{\varepsilon}\to v^{\varepsilon} \text{ weakly-$\ast$ in } L^{\infty}\left(0,1;H^1(0,1)\right),\\
		&\partial_xv_{n}^{\varepsilon}\to\partial_xv^{\varepsilon} \text{ weakly-$\ast$ in } L^{\infty}\left(0,1;L^2(0,1)\right).
	\end{align*}
	Let $\mathbb{S}_n^{\perp}$ is a closed subspace of $H_0^1(0,1)$ such that $H_0^1(0,1) = \mathbb{S}_n\oplus\mathbb{S}_n^{\perp}$. For all $\psi\in H_0^1(0,1)$, we can write $\psi$ of the form $\psi = \psi_n+\psi_n^{\perp}$ where $\psi_n\in \mathbb{S}_n$ and $\psi_n^{\perp}\in \mathbb{S}_n^{\perp}$. Using the Galerkin equation \eqref{Galerkinequav}, we can show that $\partial_{x^2}^2v_{n}^{\varepsilon}\in L^2(0,1;\mathbb{S}_n)$. In particular, for $\psi_n\in\mathbb{S}_n$, we have
	\begin{align*}
		\left\langle\partial_{x^2}^2v_{n}^{\varepsilon}(x,\cdot),\psi_n\right\rangle = -\left\langle\partial_yv_{n}^{\varepsilon}(x,\cdot),\partial_y\psi_n\right\rangle - 2r\left\langle \partial_x v_n^{\varepsilon}(x,\cdot),\psi_n \right\rangle- \left\langle\textbf{P}_{\varepsilon}v_{n}^{\varepsilon}(x,\cdot),\psi_n\right\rangle
		+(k^2-r^2)\left\langle v_{n}^{\varepsilon}(x,\cdot),\psi_n\right\rangle.
	\end{align*}
	Using Cauchy-Schwarz’s inequality and the fact that $\rVert{\psi_n}\rVert_{H_0^1(0,1)}\leq\left\rVert{\psi}\right\rVert_{H_0^1(0,1)}$ with $\psi = \psi_n+\psi_n^{\perp}$, we have
	\begin{align}
		\begin{split}\label{CS1}
			\left|\left\langle\partial_yv_{n}^{\varepsilon}(x,\cdot),\partial_y\psi_n\right\rangle\right|\leq \big\rVert{\partial_yv_{n}^{\varepsilon}(x,\cdot)}\big\rVert \big\rVert{\partial_y\psi_n}\big\rVert\leq \big\rVert{\partial_y v_{n}^{\varepsilon}(x,\cdot)}\big\rVert\rVert{\psi}\rVert_{H_0^1(0,1)},
		\end{split}\\
		\begin{split}\label{CS1-1}
			\left|\left\langle \partial_{x}v_{n}^{\varepsilon}\left(x,\cdot\right),\psi_{n}\right\rangle \right|\le\left\Vert \partial_{x}v_{n}^{\varepsilon}\left(x,\cdot\right)\right\Vert \left\Vert \psi_{n}\right\Vert \le\left\Vert \partial_{x}v_{n}^{\varepsilon}\left(x,\cdot\right)\right\Vert \left\Vert \psi\right\Vert _{H_{0}^{1}\left(0,1\right)},
		\end{split}\\
		\begin{split}\label{CS2}
			\left|\left\langle\textbf{P}_{\varepsilon}v_{n}^{\varepsilon}(x,\cdot),\psi_n\right\rangle\right|\leq\big\rVert{\textbf{P}_{\varepsilon}v_{n}^{\varepsilon}(x,\cdot)}\big\rVert
			\rVert{\psi_n}\big\rVert\leq 2\log\left(\gamma\right)\left\rVert{\partial_y v_{n}^{\varepsilon}(x,\cdot)}\right\rVert\left\rVert{\psi}\right\rVert_{H_0^1(0,1)},
		\end{split}\\
		\begin{split}\label{CS3}
			\left|\left\langle v_{n}^{\varepsilon}(x,\cdot),\psi_n\right\rangle\right|\leq\left\rVert{v_{n}^{\varepsilon}(x,\cdot)}\right\rVert\left\rVert{\psi_n}\right\rVert\leq \big\rVert{v_{n}^{\varepsilon}(x,\cdot)}\big\rVert\rVert{\psi}\rVert_{H_0^1(0,1)}.
		\end{split}
	\end{align}
	Thus, combining the above four estimates \eqref{CS1}, \eqref{CS1-1}, \eqref{CS2} and \eqref{CS3}, we get
	\begin{align*}
		& \left\Vert \partial_{x^{2}}^{2}v_{n}^{\varepsilon}\left(x,\cdot\right)\right\Vert _{H^{-1}\left(0,1\right)}=\sup_{\psi\in H^{1}\left(0,1\right)\backslash\left\{ 0\right\} }\frac{\left\langle \partial_{x^{2}}^{2}v_{n}^{\varepsilon}\left(x,\cdot\right),\psi\right\rangle }{\left\Vert \psi\right\Vert _{H_{0}^{1}\left(0,1\right)}}\\
		& =\sup_{\psi\in H^{1}\left(0,1\right)\backslash\left\{ 0\right\} }\frac{-\left\langle \partial_{y}v_{n}^{\varepsilon}(x,\cdot),\partial_{y}\psi_{n}\right\rangle -2r\left\langle \partial_{x}v_{n}^{\varepsilon}(x,\cdot),\psi_{n}\right\rangle -\left\langle \textbf{P}_{\varepsilon}v_{n}^{\varepsilon}(x,\cdot),\psi_{n}\right\rangle +(k^{2}-r^{2})\left\langle v_{n}^{\varepsilon}(x,\cdot),\psi_{n}\right\rangle }{\left\Vert \psi\right\Vert _{H_{0}^{1}\left(0,1\right)}}\\
		& \le C\left[\left\Vert \partial_{x}v_{n}^{\varepsilon}\left(x,\cdot\right)\right\Vert +\left\Vert \partial_{y}v_{n}^{\varepsilon}\left(x,\cdot\right)\right\Vert +\left(r^{2}-k^{2}\right)\left\Vert v_{n}^{\varepsilon}\left(x,\cdot\right)\right\Vert \right].
	\end{align*}
	Henceforth, we can find a constant $C>0$ independent of $n$ to bound the $H^{-1}$ norm of $\partial_{x^2}^2 v_{n}^{\varepsilon}$ in the following manner:
	\begin{align}
		& \left\Vert \partial_{x^{2}}^{2}v_{n}^{\varepsilon}\left(x,\cdot\right)\right\Vert _{H^{-1}\left(0,1\right)}=\sup_{\psi\in H^{1}\left(0,1\right)\backslash\left\{ 0\right\} }\frac{\left\langle \partial_{x^{2}}^{2}v_{n}^{\varepsilon}\left(x,\cdot\right),\psi\right\rangle }{\left\Vert \psi\right\Vert _{H_{0}^{1}\left(0,1\right)}} \nonumber \\
		& =\sup_{\psi\in H^{1}\left(0,1\right)\backslash\left\{ 0\right\} }\frac{-\left\langle \partial_{y}v_{n}^{\varepsilon}(x,\cdot),\partial_{y}\psi_{n}\right\rangle -2r\left\langle \partial_{x}v_{n}^{\varepsilon}(x,\cdot),\psi_{n}\right\rangle -\left\langle \textbf{P}_{\varepsilon}v_{n}^{\varepsilon}(x,\cdot),\psi_{n}\right\rangle +(k^{2}-r^{2})\left\langle v_{n}^{\varepsilon}(x,\cdot),\psi_{n}\right\rangle }{\left\Vert \psi\right\Vert _{H_{0}^{1}\left(0,1\right)}} \nonumber \\
		& \le C\left(\left\Vert \partial_{x}v_{n}^{\varepsilon}\left(x,\cdot\right)\right\Vert +\left\Vert v_{n}^{\varepsilon}\left(x,\cdot\right)\right\Vert _{H^{1}\left(0,1\right)}\right). \label{vinf}
	\end{align}
	Square the above estimate, integrate the resulting estimate with respect to $x$ and then apply \eqref{Gronwallesti}. By the Banach–Alaoglu theorem, we can choose a subsequence of $v_{n}^{\varepsilon}$ so that
	\begin{align*}
		\partial_{x^2}^2 v_{n}^{\varepsilon}\to \partial_{x^2}^2 v^{\varepsilon} \text{ weakly in } L^{2}\left(0,1;H^{-1}(0,1)\right).
	\end{align*}
	Now, we combine the above weak-star and weak limits to conclude that the limit function $v^{\varepsilon}$ satisfies
	\begin{align}\label{uepschareac}
		v^{\varepsilon}\in L^{\infty}\left(0,1;H_0^1(0,1)\right),\quad \partial_x v^{\varepsilon}\in L^{\infty}\left(0,1;L^2(0,1)\right),\quad \partial_{x^2}^2 v^{\varepsilon}\in L^2\left(0,1,H^{-1}(0,1)\right),
	\end{align}
	which, by back-substitution $u_{n}^{\varepsilon}=e^{rx}v_n^{\varepsilon}$, leads to
	\begin{align}\label{uepschareacu}
		u^{\varepsilon}\in L^{\infty}\left(0,1;H_0^1(0,1)\right),\quad \partial_x u^{\varepsilon}\in L^{\infty}\left(0,1;L^2(0,1)\right),\quad \partial_{x^2}^2 u^{\varepsilon}\in L^2\left(0,1,H^{-1}(0,1)\right).
	\end{align}
	Note that the first and second properties in \eqref{uepschareacu} are obtained directly from \eqref{uepschareac}. Meanwhile, the last property in \eqref{uepschareacu} can be deduced from  \eqref{Galerkinequa} using the same token as in \eqref{vinf}. Moreover, using the Aubin-Lions lemma and the Rellich-Kondrachov embedding theorem $H_0^1(0,1) \subset L^2(0,1)$ for the first  and second properties in \eqref{uepschareacu}, we find that
	\begin{align}\label{stronglyinL2}
		u^{\varepsilon}_{n}\to u^{\varepsilon} \text{ strongly in } L^2\left(0,1;H_0^1(0,1)\right).
	\end{align}
	Now, we multiply both sides of the Galerkin equation \eqref{Galerkinequa} by an $x$-dependent test function $\tilde{w}\in C_c^{\infty}(0,1)$, then by integrate the resulting equation with respect to $x$ to get
	\begin{align*}
		\int_0^1\left\langle \frac{\partial^{2}u_{n}^{\varepsilon}}{\partial x^{2}},\nu\right\rangle dx +\int_0^1\left\langle \frac{\partial u_{n}^{\varepsilon}}{\partial y},\frac{\partial\nu}{\partial y}\right\rangle dx +\int_0^1\left\langle \mathbf{P}_{\varepsilon}u_{n}^{\varepsilon},\nu\right\rangle dx = k^{2}\int_0^1\left\langle u_{n}^{\varepsilon},\nu\right\rangle dx.
	\end{align*}
	where we have denoted by $\nu = \nu(x,y) = \tilde{w}(x)\psi(y)$ for $\psi \in \mathbb{S}_n$. Henceforth, we pass the limit of this equation as $n\to \infty$ and obtain
	\begin{align}\label{passinglimit}
		\int_0^1\left\langle \frac{\partial^{2}u^{\varepsilon}}{\partial x^{2}},\nu\right\rangle dx +\int_0^1\left\langle\frac{\partial u^{\varepsilon}}{\partial y},\frac{\partial\nu}{\partial y}\right\rangle dx +\int_0^1\left\langle \mathbf{P}_{\varepsilon}u^{\varepsilon},\nu\right\rangle dx = k^{2}\int_0^1\left\langle u^{\varepsilon},\nu\right\rangle dx.
	\end{align}
	The convergence of the second, third and fourth terms in the limit equation \eqref{passinglimit} is deduced using \eqref{stronglyinL2}. We remark that the limit equation \eqref{passinglimit} holds for $\nu = \tilde{w}\psi$ with $\psi\in H_0^1(0,1)$. In addition, since $\tilde{w}\in C_c^{\infty}(0,1)$ is arbitrary, our function $u^{\varepsilon}$ obtained from approximate solutions $u_{n}^{\varepsilon}$ satisfies the weak formulation \eqref{Galerkinequa} for every test function $\psi\in H_0^1(0,1)$. Besides, exploiting the Aubin-Lions lemma and the Gelfand triple $H_0^1(0,1)\subset L^2(0,1)\subset H^{-1}(0,1)$,  \eqref{uepschareacu} gives
	\begin{align}\label{uepschareac1}
		u^{\varepsilon}\in C\left([0,1];H_0^1(0,1)\right),\quad \partial_xu^{\varepsilon}\in C\left([0,1],L^2(0,1)\right).
	\end{align}
	Next, we verify the initial data by the following arguments. We take an arbitrary $x$-dependent function $\kappa\in C^1\left([0,1]\right)$ satisfying $\kappa(0) = 1$ and $\kappa(1) = 0$. By the second argument in \eqref{uepschareacu}, we have
	\begin{align*}
		\int_0^1\left\langle\partial_xu_{n}^{\varepsilon},\psi\right\rangle\kappa(x)dx \to \int_0^1\left\langle\partial_xu^{\varepsilon},\psi\right\rangle\kappa(x)dx \text{ for } \psi \in L^2(0,1).
	\end{align*}
	Then using integration by parts, we arrive at
	\begin{align*}
		-\left\langle u_{n}^{\varepsilon}(0),\psi\right\rangle\kappa(0)-\int_0^1\left\langle u_{n}^{\varepsilon},\psi\right\rangle\kappa_xdx \to -\left\langle u^{\varepsilon}(0),\psi\right\rangle\kappa(0)-\int_0^1\left\langle u^{\varepsilon},\psi\right\rangle\kappa_xdx.
	\end{align*}
	Henceforth, by the first argument in \eqref{uepschareacu}, we obtain the limit $\left\langle u_{n}^{\varepsilon}(0),\psi\right\rangle\to\left\langle u^{\varepsilon}(0),\psi\right\rangle$ for all $\psi\in H_0^1(0,1)$. By the strong $H^1$ convergence of $u_n^{\varepsilon}$ designated in \eqref{stronginH1}, we obtain $\left\langle u_{n}^{\varepsilon}(0),\psi\right\rangle\to\left\langle u^{\varepsilon}_0,\psi\right\rangle$ for all $\psi\in H^1(0,1)$. By the uniqueness of limit, it holds true that $\left\langle u^{\varepsilon}(0),\psi\right\rangle=\left\langle u^{\varepsilon}_0,\psi\right\rangle$ for all $\psi\in H^1(0,1)$. Thus, $u^{\varepsilon}(0) = u_0^{\varepsilon}$ for a.e. in $(0,1)$. We complete the existence result for system \eqref{reguHelm}.
	
	Now, let $u_1^{\varepsilon}$ and $u_2^{\varepsilon}$ be two weak solutions to system \eqref{reguHelm} that we have obtained in the above part. Consider $d^{\varepsilon} = e^{-rx}\left(u_1^{\varepsilon} - u_2^{\varepsilon}\right)$. Similar to \eqref{Galerkinequav}, $d^{\varepsilon}$ satisfies the following wave equation:
	\begin{align}\label{Galerkinequad}
		\left\langle \frac{\partial^{2}d^{\varepsilon}}{\partial x^{2}},\psi\right\rangle +\left\langle \frac{\partial d^{\varepsilon}}{\partial y},\frac{\partial\psi}{\partial y}\right\rangle + 2r\left\langle \frac{\partial d^{\varepsilon}}{\partial x},\psi  \right\rangle + \left\langle \mathbf{P}_{\varepsilon}d^{\varepsilon},\psi\right\rangle = (k^{2} - r^{2})\left\langle d^{\varepsilon},\psi\right\rangle \quad \text{for } \psi\in H_0^1(0,1).
	\end{align}
	Taking in \eqref{Galerkinequad} $\psi = \partial_xd^{\varepsilon}$, we follow the same process of getting \eqref{Gronwallesti}. Thus, we derive that
	\begin{align}
		\big\rVert{\partial_x d^{\varepsilon}(x,\cdot)}\big\rVert^2+\big\rVert{\partial_y d^{\varepsilon}(x,\cdot)}\big\rVert^2
		& + (r^2 -k^2)\big\rVert{d^{\varepsilon}(x,\cdot)}\big\rVert^2 \nonumber \\ & 
		\leq \left(\big\rVert{\partial_x d^{\varepsilon}(0,\cdot)}\big\rVert^2+ \big\rVert{\partial_y d^{\varepsilon}(0,\cdot)}\big\rVert^2+(r^2 -k^2)\big\rVert{d^{\varepsilon}(0,\cdot)}\big\rVert^2\right)\gamma^{2x}.\label{estd}
	\end{align}
	Since $u_1^{\varepsilon}$ and $u_2^{\varepsilon}$ have the same boundary and initial data, we find that
	\begin{align*}
		& d^{\varepsilon}\left(0,y\right)=u_{1}^{\varepsilon}\left(0,y\right)-u_{2}^{\varepsilon}\left(0,y\right)=0,\\
		& \partial_{x}d^{\varepsilon}\left(0,y\right)=-rd^{\varepsilon}\left(0,y\right)+\partial_{x}u_{1}^{\varepsilon}\left(0,y\right)-\partial_{x}u_{2}^{\varepsilon}\left(0,y\right) = 0,\\
		& \partial_{y}d^{\varepsilon}\left(0,y\right)=\partial_{y}u_{1}^{\varepsilon}\left(0,y\right)-\partial_{y}u_{2}^{\varepsilon}\left(0,y\right)=0.
	\end{align*}
	This shows that the left-hand side of \eqref{estd} is non-positive for $r>k$, which indicates the uniqueness result for  system \eqref{reguHelm}. Hence, we complete the proof of the theorem.
\end{proof}
	
It is worth mentioning that the weight $e^{-rx}$ is employed in the proof of Theorem \ref{thm:weak}. Commonly, this is called \textit{Carleman weight}, playing a vital role not only in prove the existence and uniqueness results, but also in convergence estimates of regularization schemes for inverse and ill-posed problems. This is manifested in the present PDE-approach as well as its variant for ill-posed parabolic problems; cf. \cite{Nguyen2019,Khoa2020a}).  In the so-called convexification method, which is a Tikhonov-like regularization technique, the Carleman weight is used to ``convexify'' nonlinear cost functionals to obtain a unique minimizer; cf. e.g. \cite{Klibanov2013,Khoa2020,Klibanov2022,Le2022}. The use of the smooth weight $e^{-rx}$ in the present work is based on the following reasons. First, it maximizes the presence of initial data since the weight is exponentially decreasing. Second, it helps to control large stability estimate of the stabilized operator (i.e. the term $\log(\gamma)$ as $\gamma\to \infty$) as well as the presence of terms involving $k$ that negatively affect the energy estimates. Below, we continue to apply the Carleman weight to prove the distance between regularized solution $u^{\varepsilon}$ and true solution $u$ in Theorem \ref{thm:4-1}. Then, convergence results follow.

\begin{theorem}[Rigorous mixed error estimates]\label{thm:4-1}
	Let $u\in C\left([0,1];H^2(0,1)\right)\cap C^1\left([0,1];H^1(0,1)\right)$ be a unique solution of the Cauchy problem  \eqref{helmholtz}--\eqref{boundary1}. Let $M>0$ independent of $\varepsilon$ and $k$ be such that the true solution satisfies $\left\Vert u\right\Vert_{C\left([0,1];H^2(0,1)\right)\cap C^1\left([0,1];H^1(0,1)\right)} \le M$. Let $u^{\varepsilon}$ be a unique weak solution of system \eqref{reguHelm} as defined in Definition \ref{def:weak1} and analyzed in Theorem \ref{thm:weak}.  Assume that $\log(\gamma)\ge k$ holds true. Then, the following mixed $L^2$-$H^1$ error estimates hold true for any $\rho > k$:
	\begin{align}
		& \left\rVert u^{\varepsilon}(x,\cdot)-u(x,\cdot)\right\rVert ^{2}\leq\left[\dfrac{\left(\rho^{2}+1\right)\varepsilon^{2}}{\rho^{2}-k^{2}}+\varepsilon^{2}+\dfrac{\left(1-e^{-2\rho x}\right)\rho^{-1}M^{2}\gamma^{-2}}{8k\left(\rho^{2}-k^{2}\right)}\right]\gamma^{2x}e^{2\rho x}, \label{mainest1}\\
		& \left\rVert u_{y}^{\varepsilon}(x,\cdot)-u_{y}(x,\cdot)\right\rVert ^{2}\leq\left[\left(\rho^{2}+1\right)\varepsilon^{2}+\varepsilon^{2}\left(\rho^{2}-k^{2}\right)+\dfrac{\left(1-e^{-2\rho x}\right)\rho^{-1}M^{2}\gamma^{-2}}{8k}\right]\gamma^{2x}e^{2\rho x}, \label{mainest2}\\
		& \left\Vert u_{x}^{\varepsilon}\left(x,\cdot\right)-u_{x}\left(x,\cdot\right)\right\Vert ^{2} \nonumber \\
		& \le\left[\left(\rho^{2}+1\right)\varepsilon^{2}+\varepsilon^{2}\left(\rho^{2}-k^{2}\right)+\dfrac{\left(1-e^{-2\rho x}\right)\rho^{-1}M^{2}\gamma^{-2}}{8k}\right]\gamma^{2x}e^{2\rho x}\left[e^{\rho x}+\frac{\rho}{\rho^{2}-k^{2}}\right]^{2}.\label{mainest3}
	\end{align}
\end{theorem}
\begin{proof}
	Let $w = e^{-\rho x}\left(u^{\varepsilon}-u\right)$ where $\rho>0$ is a constant needed to be chosen latter. From \eqref{eq:4} and \eqref{reguHelm}, we are capable of computing the difference equation. In particular, the diffence function $w$ satisfies the following damped wave equation:
	\begin{align}\label{diffequa}
		w_{xx} - w_{yy} + 2\rho w_x + (\rho^2 - k^2)w= -\textbf{P}_{\varepsilon}w - e^{-\rho x}\textbf{Q}_{\varepsilon}u.
	\end{align}
	This equation is associated with the Dirichlet boundary condition and the initial conditions
	\begin{align}
		\begin{cases}
			w(x,0) = w(x,1) = 0 &\text{ for } x\in[0,1],\\
			w(0,y) = u_0^{\varepsilon}(y) - u_0(0),\quad  \partial_xw(0,y) = -\rho w(0,y) &\text{ for } x\in[0,1].
		\end{cases}
	\end{align}
	Multiplying both sides of \eqref{diffequa} by $w_x$ and integrating the resulting equation with respect to $y$ from 0 to 1, we have
	\begin{align}
		\begin{split}
			\dfrac{d}{dx}\left\rVert{w_x(x,\cdot)}\right\rVert^2&+\dfrac{d}{dx}\left\rVert{w_y(x,\cdot)}\right\rVert^2+4\rho\left\rVert{w_x(x,\cdot)}\right\rVert^2\\
			&+(\rho^2-k^2)\dfrac{d}{dx}\left\rVert{w(x,\cdot)}\right\rVert^2 = -2\left\langle\textbf{P}_{\varepsilon}w(x,\cdot),w_x(x,\cdot)\right\rangle-2e^{-\rho x}\left\langle\textbf{Q}_{\varepsilon}u(x,\cdot),w_x(x,\cdot)\right\rangle,
		\end{split}
	\end{align}
or equivalently,
	\begin{align}\label{wxdiffequa}
		\begin{split}
			\dfrac{1}{\rho^2 - k^2}&\dfrac{d}{dx}\left(\left\rVert{w_x(x,\cdot)}\right\rVert^2+\left\rVert{w_y(x,\cdot)}\right\rVert^2\right)+\dfrac{4\rho}{\rho^2-k^2}\left\rVert{w_x(x,\cdot)}\right\rVert^2\\
			+&\dfrac{d}{dx}\left\rVert{w(x,\cdot)}\right\rVert^2 = -\dfrac{2}{\rho^2-k^2}\left\langle\textbf{P}_{\varepsilon}w(x,\cdot),w_x(x,\cdot)\right\rangle - \dfrac{2e^{-\rho x}}{\rho^2-k^2}\left\langle\textbf{Q}_{\varepsilon}u(x,\cdot),w_x(x,\cdot)\right\rangle.
		\end{split}
	\end{align}
	Now, we use the energy estimates of the perturbing and stabilized operatored deduced in \eqref{estQ}, \eqref{Q2}, \eqref{estP}. Then applying the Cauchy--Schwarz inequality, we estimate two terms in the right-hand side of \eqref{wxdiffequa} as follows:
	\begin{align*}
		\dfrac{2}{\rho^2-k^2}\left|\left\langle\textbf{P}_{\varepsilon}w(x,\cdot),w_x(x,\cdot)\right\rangle\right| &\leq \dfrac{2}{\rho^2-k^2}\left\rVert{\textbf{P}_{\varepsilon}w(x,\cdot)}\right\rVert \left\rVert{w_x(x,\cdot)}\right\rVert \leq \dfrac{4\log(\gamma)}{\rho^2-k^2}\left\rVert{w_y(x,\cdot)}\right\rVert\left\rVert{w_x(x,\cdot)}\right\rVert\\
		&\leq \dfrac{2\log(\gamma)}{\rho^2-k^2}\left\rVert{w_y(x,\cdot)}\right\rVert^2+\dfrac{2\log(\gamma)}{\rho^2-k^2}\left\rVert{w_x(x,\cdot)}\right\rVert^2,\\
		\dfrac{2e^{-\rho x}}{\rho^2-k^2}\left|\left\langle\textbf{Q}_{\varepsilon}u(x,\cdot),w_x(x,\cdot)\right\rangle\right|&\leq\dfrac{2e^{-\rho x}}{\rho^2- k^2}\left\rVert{\textbf{Q}_{\varepsilon}u(x,\cdot)}\right\rVert\left\rVert{w_x(x,\cdot)}\right\rVert\leq \dfrac{e^{-2\rho x}M^2\gamma^{-2}}{4k\left(\rho^2-k^2\right)}+\dfrac{4k\left\rVert{w_x(x,\cdot)}\right\rVert^2}{\rho^2-k^2}.
	\end{align*}
	Integrating \eqref{diffequa} with respect to $x$ from 0 to $\xi$, then the left-hand side of \eqref{wxdiffequa} is bounded from above by
	\begin{align}\label{diffinequa1}
		\begin{split}
			&\dfrac{1}{\rho^2 - k^2}\left(\left\rVert{w_x(\xi,\cdot)}\right\rVert^2+\left\rVert{w_y(\xi,\cdot)}\right\rVert^2\right) + \left\rVert{w(\xi,\cdot)}\right\rVert^2\\
			&
			\le  \dfrac{1}{\rho^2 - k^2}\left(\left\rVert{w_x(0,\cdot)}\right\rVert^2+\left\rVert{w_y(0,\cdot)}\right\rVert^2\right) +\left\rVert{w(0,\cdot)}\right\rVert^2\\
			&+\int_0^{\xi}\left[\dfrac{2\log(\gamma)}{\rho^2-k^2}\left\rVert{w_y(x,\cdot)}\right\rVert^2+\dfrac{2\log(\gamma)}{\rho^2-k^2}\left\rVert{w_x(x,\cdot)}\right\rVert^2+\dfrac{e^{-2\rho x}M^2\gamma^{-2}}{4k\left(\rho^2-k^2\right)}\right]dx
			+\int_0^{\xi}\dfrac{4k-4\rho}{\rho^2-k^2}\left\rVert{w_x}(x,\cdot)\right\rVert^2dx.
		\end{split}
	\end{align}
	In \eqref{diffinequa1}, we choose $\rho$ arbitrarily such that $\rho > k$. Furthermore, at $x = 0$, the difference function $w$ and its gradients are bounded, according to \eqref{noisecondition}, by
	\begin{align*}
		\dfrac{1}{\rho^2 - k^2}&\left(\left\rVert{w_x(0,\cdot)}\right\rVert^2+\left\rVert{w_y(0,\cdot)}\right\rVert^2\right) + \left\rVert{w(0,\cdot)}\right\rVert^2 \\
		&\leq \dfrac{1}{\rho^2 - k^2}\left(\rho^2\left\rVert{u_0^{\varepsilon}-u_0}\right\rVert^2+\left\rVert{\partial_yu_0^{\varepsilon}-\partial_yu_0}\right\rVert^2\right) + \left\rVert{u_0^{\varepsilon}-u_0}\right\rVert^2\leq \dfrac{\left(\rho^2+1\right)\varepsilon^2}{\rho^2 - k^2} + \varepsilon^2.
	\end{align*}
	Henceforth, we continue to estimate the left-hand side of \eqref{diffinequa1} as follows:
	\begin{align}\label{diffinequa2}
		\begin{split}
		 \dfrac{1}{\rho^2 - k^2}&\left(\left\rVert{w_x(\xi,\cdot)}\right\rVert^2+\left\rVert{w_y(\xi,\cdot)}\right\rVert^2\right)  + \left\rVert{w(\xi,\cdot)}\right\rVert^2  \leq \dfrac{\left(\rho^2+1\right)\varepsilon^2}{\rho^2 - k^2} + \varepsilon^2+\dfrac{\left(1-e^{-2\rho\xi}\right)\rho^{-1}M^2\gamma^{-2}}{8k\left(\rho^2-k^2\right)}\\
			&+2\log(\gamma)\int_0^{\xi}\left[\dfrac{1}{\rho^2 - k^2}\left(\left\rVert{w_x(x,\cdot)}\right\rVert^2+\left\rVert{w_y(x,\cdot)}\right\rVert^2\right) + \left\rVert{w(x,\cdot)}\right\rVert^2\right]dx.
		\end{split}
	\end{align}
	Thus, using Gronwall’s inequality we obtain
	\begin{align}\label{est1}
		\dfrac{1}{\rho^2 - k^2}\left(\left\rVert{w_x(\xi,\cdot)}\right\rVert^2+\left\rVert{w_y(\xi,\cdot)}\right\rVert^2\right) &+ \left\rVert{w(\xi,\cdot)}\right\rVert^2 \leq \left[\dfrac{\left(\rho^2+1\right)\varepsilon^2}{\rho^2 - k^2} + \varepsilon^2+\dfrac{\left(1-e^{-2\rho\xi}\right)\rho^{-1}M^2\gamma^{-2}}{8k\left(\rho^2-k^2\right)}\right]\gamma^{2\xi}.
	\end{align}
	We are now in a great position to deduce the error estimate by back-substitution $w = e^{-\rho x}\left(u^{\varepsilon}-u\right)$. Dropping the gradient terms in the left-hand side of \eqref{est1}, we arrive at
	\begin{align}\label{est2}
		\left\rVert{u^{\varepsilon}(\xi,\cdot) - u(\xi,\cdot)}\right\rVert^2 \leq \left[\dfrac{\left(\rho^2+1\right)\varepsilon^2}{\rho^2 - k^2} + \varepsilon^2+\dfrac{\left(1-e^{-2\rho\xi}\right)\rho^{-1}M^2\gamma^{-2}}{8k\left(\rho^2-k^2\right)}\right]\gamma^{2\xi}e^{2\rho \xi}.
	\end{align}
	Similarly, dropping the first and third terms in the left-hand side of \eqref{est1}, we get
	\begin{align*}
		\left\rVert{u^{\varepsilon}_{y}(\xi,\cdot) - u_{y}(\xi,\cdot)}\right\rVert^2 \leq \left[\left(\rho^2+1\right)\varepsilon^2 + \varepsilon^2 \left(\rho^2-k^2\right)+\dfrac{\left(1-e^{-2\rho\xi}\right)\rho^{-1}M^2\gamma^{-2}}{8k}\right]\gamma^{2\xi}e^{2\rho \xi}.
	\end{align*}
	In view of the fact that $w_x = -\rho w + e^{-\rho x} \left(u^{\varepsilon}_x - u_x\right)$, we find that
	\begin{align*}
		e^{-\rho\xi}\left\Vert u_{x}^{\varepsilon}\left(\xi,\cdot\right)-u_{x}\left(\xi,\cdot\right)\right\Vert  & -\rho\left\Vert w\left(\xi,\cdot\right)\right\Vert \le\left\Vert w_{x}\left(\xi,\cdot\right)\right\Vert \\
		& \le\left[\left(\rho^{2}+1\right)\varepsilon^{2}+\varepsilon^{2}\left(\rho^{2}-k^{2}\right)+\dfrac{\left(1-e^{-2\rho\xi}\right)\rho^{-1}M^{2}\gamma^{-2}}{8k}\right]^{1/2}\gamma^{\xi}e^{\rho\xi}.
	\end{align*}
	Combining this with \eqref{est2} leads to
	\begin{align*}
		& \left\Vert u_{x}^{\varepsilon}\left(\xi,\cdot\right)-u_{x}\left(\xi,\cdot\right)\right\Vert \\
		& \le\left[\left(\rho^{2}+1\right)\varepsilon^{2}+\varepsilon^{2}\left(\rho^{2}-k^{2}\right)+\dfrac{\left(1-e^{-2\rho\xi}\right)\rho^{-1}M^{2}\gamma^{-2}}{8k}\right]^{1/2}\gamma^{\xi}e^{2\rho\xi}+\rho\left\Vert u^{\varepsilon}\left(\xi,\cdot\right)-u\left(\xi,\cdot\right)\right\Vert \\
		& \le\left[\left(\rho^{2}+1\right)\varepsilon^{2}+\varepsilon^{2}\left(\rho^{2}-k^{2}\right)+\dfrac{\left(1-e^{-2\rho\xi}\right)\rho^{-1}M^{2}\gamma^{-2}}{8k}\right]^{1/2}\gamma^{\xi}e^{\rho\xi}\left[e^{\rho\xi}+\frac{\rho}{\rho^{2}-k^{2}}\right].
	\end{align*}
	This is equivalent to
	\begin{align*}
		& \left\Vert u_{x}^{\varepsilon}\left(\xi,\cdot\right)-u_{x}\left(\xi,\cdot\right)\right\Vert ^{2}\\
		& \le\left[\left(\rho^{2}+1\right)\varepsilon^{2}+\varepsilon^{2}\left(\rho^{2}-k^{2}\right)+\dfrac{\left(1-e^{-2\rho\xi}\right)\rho^{-1}M^{2}\gamma^{-2}}{8k}\right]\gamma^{2\xi}e^{2\rho\xi}\left[e^{\rho\xi}+\frac{\rho}{\rho^{2}-k^{2}}\right]^{2}.
	\end{align*}
	Hence, we complete the proof of the theorem.
\end{proof}

As a consequence of the above theorem, one can obtain the Lipschitz stability of $u^{\varepsilon}$. Indeed, let $v^{\varepsilon}$ be a solution to the same system as \eqref{reguHelm}, i.e.,
\begin{align*}
	\begin{cases}
		v_{xx}^{\varepsilon}-v_{yy}^{\varepsilon}+\mathbf{P}_{\varepsilon}v^{\varepsilon} =k^2 v^{\varepsilon} & \text{in }\Omega,\\
		v^{\varepsilon}\left(x,0\right)=v^{\varepsilon}\left(x,1\right)=0 & \text{for }x\in\left(0,1\right),\\
		v^{\varepsilon}\left(0,y\right)=v_{0}^{\varepsilon}\left(y\right),\quad v_{x}^{\varepsilon}\left(0,y\right)=0 & \text{for }y\in\left(0,1\right).
	\end{cases}
\end{align*}
Then, proceeding as in the proof of Theorem \ref{thm:4-1} (without the presence of $\mathbf{Q}_{\varepsilon}$), we can prove the following Lipschitz stability estimates:
\begin{align*}
	& \left\rVert u^{\varepsilon}(x,\cdot)-v^{\varepsilon}(x,\cdot)\right\rVert ^{2}\leq\left[\dfrac{\rho^{2}\left\rVert u_{0}^{\varepsilon}-v_{0}^{\varepsilon}\right\rVert ^{2}+\left\rVert \partial_{y}u_{0}^{\varepsilon}-\partial_{y}v_{0}^{\varepsilon}\right\rVert ^{2}}{\rho^{2}-k^{2}}+\left\rVert u_{0}^{\varepsilon}-v_{0}^{\varepsilon}\right\rVert ^{2}\right]\gamma^{2x}e^{2\rho x},\\
	& \left\rVert u_{y}^{\varepsilon}(x,\cdot)-v_{y}^{\varepsilon}(x,\cdot)\right\rVert ^{2}\leq\left[\rho^{2}\left\rVert u_{0}^{\varepsilon}-v_{0}^{\varepsilon}\right\rVert ^{2}+\left\rVert \partial_{y}u_{0}^{\varepsilon}-\partial_{y}v_{0}^{\varepsilon}\right\rVert ^{2}+\left(\rho^{2}-k^{2}\right)\left\rVert u_{0}^{\varepsilon}-v_{0}^{\varepsilon}\right\rVert ^{2}\right]\gamma^{2x}e^{2\rho x},\\
	& \left\Vert u_{x}^{\varepsilon}\left(x,\cdot\right)-v_{x}^{\varepsilon}\left(x,\cdot\right)\right\Vert ^{2}\\
	& \le\left[\rho^{2}\left\rVert u_{0}^{\varepsilon}-v_{0}^{\varepsilon}\right\rVert ^{2}+\left\rVert \partial_{y}u_{0}^{\varepsilon}-\partial_{y}v_{0}^{\varepsilon}\right\rVert ^{2}+\left(\rho^{2}-k^{2}\right)\left\rVert u_{0}^{\varepsilon}-v_{0}^{\varepsilon}\right\rVert ^{2}\right]\gamma^{2x}e^{2\rho x}\left[e^{\rho x}+\frac{\rho}{\rho^{2}-k^{2}}\right]^{2}.
\end{align*}
The estimates are valid for any $\gamma$ satisfying $\log(\gamma) \ge k$. Now, to prove the convergence of $u^{\varepsilon}$ toward $u$, we below rely on a suitable choice of $\gamma$. In this regard, $\gamma$ is appropriately dependent of the noise level $\varepsilon$. Since Theorem \ref{thm:4} below is a direct consequence of Theorem \ref{thm:4-1}, its proof is omitted.


\begin{theorem}[Interior convergence estimates]\label{thm:4}
	Under the assumptions of Theorem \ref{thm:4-1}, if we choose $\gamma = \varepsilon^{-\alpha}$ for $\alpha \in (0,1]$ and $\varepsilon\le e^{-k/\alpha}$, then the following H\"older rates of convergence hold true:
	\begin{align} 
		& \left\rVert u^{\varepsilon}(x,\cdot)-u(x,\cdot)\right\rVert ^{2}\leq\left[\dfrac{\left(\rho^{2}+1\right)\varepsilon^{2\left(1-\alpha x\right)}}{\rho^{2}-k^{2}}+\varepsilon^{2\left(1-\alpha x\right)}+\dfrac{\left(1-e^{-2\rho x}\right)\rho^{-1}M^{2}\varepsilon^{2\alpha\left(1-x\right)}}{8k\left(\rho^{2}-k^{2}\right)}\right]e^{2\rho x}, \label{nextest1}\\
		& \left\rVert u_{y}^{\varepsilon}(x,\cdot)-u_{y}(x,\cdot)\right\rVert ^{2}\leq\left[\left(\rho^{2}+1\right)\varepsilon^{2\left(1-\alpha x\right)}+\left(\rho^{2}-k^{2}\right)\varepsilon^{2\left(1-\alpha x\right)}+\dfrac{\left(1-e^{-2\rho x}\right)\rho^{-1}M^{2}\varepsilon^{2\alpha\left(1-x\right)}}{8k}\right]e^{2\rho x},\label{nextest2}\\
		& \left\Vert u_{x}^{\varepsilon}\left(x,\cdot\right)-u_{x}\left(x,\cdot\right)\right\Vert ^{2} \nonumber\\
		& \le\left[\left(\rho^{2}+1\right)\varepsilon^{2\left(1-\alpha x\right)}+\left(\rho^{2}-k^{2}\right)\varepsilon^{2\left(1-\alpha x\right)}+\dfrac{\left(1-e^{-2\rho x}\right)\rho^{-1}M^{2}\varepsilon^{2\alpha\left(1-x\right)}}{8k}\right]e^{2\rho x}\left[e^{\rho x}+\frac{\rho}{\rho^{2}-k^{2}}\right]^{2}.\label{nextest3}
	\end{align}
\end{theorem}

It is straightforward that regardless of the choice of $\gamma$, we do not have the convergence at $x=1$ due to the term $\gamma^{2-2x}$ in \eqref{mainest1}--\eqref{mainest3}. The same can be manifested in \eqref{nextest1}--\eqref{nextest3}. For the interior points $x\in (0,1)$, we obtain the H\"older rate of convergence in mixed $L^2$--$H^1$ norms. We also remark that even though convergence result \eqref{nextest1} is point wise in the frequency $k$, the corresponding uniform-in-$k$ estimate can be obtained by a suitable choice of $\rho$. Observing the exponential growth (in $\rho$) of the inverse Carleman weight, $\rho$ should be close to $k$ to ``optimize'' that growth. However, the closeness should also ensure the $L^2$ convergence in $\varepsilon$. Thus, one possibility is taking $\rho = k \log\left(\log\left(\varepsilon^{-\beta}\right)\right)$ for $\beta\in(0,1)$ and $\varepsilon$ being sufficiently small such that $\log\left(\log\left(\varepsilon^{-\beta}\right)\right)\ge \sqrt{2}$. By this way, $1<\rho^2 \le 2(\rho^2 - k^2)$ and thus, it follows from $\varepsilon^{2\alpha(1-x)}\ge \varepsilon^{2(1-\alpha x)}$ that
\begin{align}\label{uniformk}
\left\rVert u^{\varepsilon}(x,\cdot)-u(x,\cdot)\right\rVert ^{2}\leq C^2\varepsilon^{2\alpha\left(1-x\right)}\log^{2kx}\left(\varepsilon^{-\beta}\right),
\end{align}
where $C>0$ is a constant depending only on $M$. 

\begin{theorem}[Boundary convergence estimates]\label{thm:6}
	 Under the assumptions of Theorem \ref{thm:4-1}, we can always find $x_{\varepsilon}\in (0,1)$ such that $\lim_{\varepsilon\to 0} x_{\varepsilon} = 0$ and $u^{\varepsilon}(1-x_{\varepsilon},\cdot)$ approximates well $u(1,\cdot)$. In particular, we can find a constant $C(\rho,k,M)>0$ depending on $\rho,k,M$ such that the following logarithmic convergence estimate holds true:
	\begin{align}\label{lograte1}
		\left\Vert u^{\varepsilon}\left(1-x_{\varepsilon},\cdot\right)-u\left(1,\cdot\right)\right\Vert \le\frac{C\left(\rho,k,M\right)}{1+\sqrt{1+4\log\left(\varepsilon^{-\alpha}\right)}}.
	\end{align}
	Moreover, for $\varepsilon$ being sufficiently small such that $\log\left(\log\left(\varepsilon^{-\beta}\right)\right)\ge \sqrt{2}$, the corresponding uniform-in-$k$ error estimate can be rigorously obtained in the following form:
	\begin{align}\label{lograte2}
		\left\Vert u^{\varepsilon}\left(1-x_{\varepsilon},\cdot\right)-u\left(1,\cdot\right)\right\Vert \le\frac{C(M)}{1-k\log\left(\log\left(\varepsilon^{-\beta}\right)\right)+\sqrt{\left(1-k\log\left(\log\left(\varepsilon^{-\beta}\right)\right)\right)^{2}+4k\log\left(\log\left(\varepsilon^{-\beta}\right)\right)+4\log\left(\varepsilon^{-\alpha}\right)}},
	\end{align}
	where $C(M)>0$ is a constant depending only on $M$.
\end{theorem}
\begin{proof}
	We, for brevity, can find some constant $C(\rho,k,M)>0$ such that for $x\in (0,1)$,
	\[
	\left\rVert u^{\varepsilon}(x,\cdot)-u(x,\cdot)\right\rVert ^{2}\le C^2\left(\rho,k,M\right)\varepsilon^{2\alpha\left(1-x\right)},
	\]
	deduced from \eqref{nextest1}. Thereby, using the triangle inequality, we get
	\begin{align*}
		\left\Vert u^{\varepsilon}\left(1-x_{\varepsilon},\cdot\right)-u\left(1,\cdot\right)\right\Vert  & \le\left\Vert u^{\varepsilon}\left(1-x_{\varepsilon},\cdot\right)-u\left(1-x_{\varepsilon},\cdot\right)\right\Vert +\left\Vert u\left(1-x_{\varepsilon},\cdot\right)-u\left(1,\cdot\right)\right\Vert \\
		& \le C\left(\rho,k,M\right)\varepsilon^{\alpha x_{\varepsilon}}+x_{\varepsilon}\left\Vert u_{x}\right\Vert _{C\left(\left[0,1\right];L^{2}\left(0,1\right)\right)}.
	\end{align*}
	Therefore, to prove the target estimate \eqref{lograte1}, we seek the infimum $\frac{1}{2}\inf_{x_{\varepsilon}>0}\left(\varepsilon^{2\alpha x_{\varepsilon}} + x_{\varepsilon}\right)$. When doing so, we solve the following algebraic equation:
	\begin{align}\label{alge}
		\varepsilon^{\alpha x_{\varepsilon}} = x_{\varepsilon},
	\end{align}
	expecting that $x_{\varepsilon}\in (0,1)$ is sufficiently small. In terms of $x_{\varepsilon}$, we see that the left-hand side of \eqref{alge} is decreasing, while the right-hand side grows linearly. Thus, for every $\varepsilon>0$, there exists a unique solution $x_{\varepsilon}\in (0,1)$ to \eqref{alge}. Taking now the logarithm on both sides of  \eqref{alge} and using the standard inequality $\log(a)> 1 - a^{-1}$ for any $a>0$, we arrive at the following quadratic inequality:
	\[
	\alpha\log\left(\varepsilon\right)x_{\varepsilon}^{2}-x_{\varepsilon}+1>0.
	\]
	Since the discriminant is positive, i.e.  $1+4\log\left(\varepsilon^{-\alpha}\right)>0$, and $\log(\varepsilon)<1$, we find that
	\[
	x_{\varepsilon}\in\left(\frac{1+\sqrt{1+4\log\left(\varepsilon^{-\alpha}\right)}}{2\alpha\log\left(\varepsilon\right)},\frac{1-\sqrt{1+4\log\left(\varepsilon^{-\alpha}\right)}}{2\alpha\log\left(\varepsilon\right)}\right).
	\]
	By the rationalizing technique, it is clear that $x_{\varepsilon}\to 0$ as $\varepsilon \to 0$. In particular, we have
	\[
	\lim_{\varepsilon\to0}\left(\frac{1-\sqrt{1+4\log\left(\varepsilon^{-\alpha}\right)}}{2\alpha\log\left(\varepsilon\right)}\right)=\lim_{\varepsilon\to0}\frac{-4\log\left(\varepsilon^{-\alpha}\right)}{2\alpha\log\left(\varepsilon\right)\left(1+\sqrt{1+4\log\left(\varepsilon^{-\alpha}\right)}\right)}=\lim_{\varepsilon\to0}\frac{2}{1+\sqrt{1+4\log\left(\varepsilon^{-\alpha}\right)}}=0,
	\]
	and similarly,
	\[
	\lim_{\varepsilon\to0}\left(\frac{1+\sqrt{1+4\log\left(\varepsilon^{-\alpha}\right)}}{2\alpha\log\left(\varepsilon\right)}\right)=\lim_{\varepsilon\to0}\frac{-4\log\left(\varepsilon^{-\alpha}\right)}{2\alpha\log\left(\varepsilon\right)\left(1-\sqrt{1+4\log\left(\varepsilon^{-\alpha}\right)}\right)}=\lim_{\varepsilon\to0}\frac{2}{1-\sqrt{1+4\log\left(\varepsilon^{-\alpha}\right)}}=0.
	\]
	Henceforth, we obtain the following error estimate point-wise in $k$:
	\[
	\left\Vert u^{\varepsilon}\left(1-x_{\varepsilon},\cdot\right)-u\left(1,\cdot\right)\right\Vert \le\frac{2\left(C\left(\rho,k,M\right)+M\right)}{1+\sqrt{1+4\log\left(\varepsilon^{-\alpha}\right)}}.
	\]
	We can prove the uniform-in-$k$ error estimate \eqref{lograte2} using the same vein. To do so, we rely on the estimate we have briefly analyzed in \eqref{uniformk}. In this case, we have
	\begin{align}\label{uni}
	\left\Vert u^{\varepsilon}\left(1-x_{\varepsilon},\cdot\right)-u\left(1,\cdot\right)\right\Vert \le C\varepsilon^{\alpha x_{\varepsilon}}\log^{k\left(1-x_{\varepsilon}\right)}\left(\varepsilon^{-\beta}\right)+x_{\varepsilon}\left\Vert u_{x}\right\Vert _{C\left(\left[0,1\right];L^{2}\left(0,1\right)\right)}.
	\end{align}
	Therefore, we study the infimum $\frac{1}{2}\inf_{x_{\varepsilon}>0}\left(\varepsilon^{\alpha x_{\varepsilon}}\log^{k\left(1-x_{\varepsilon}\right)}\left(\varepsilon^{-\beta}\right)+x_{\varepsilon}\right)$ by solving the following algebraic equation: $\varepsilon^{\alpha x_{\varepsilon}}\log^{k\left(1-x_{\varepsilon}\right)}\left(\varepsilon^{-\beta}\right)=x_{\varepsilon}$. Taking the logarithm on both sides of this equation and then using the logarithmic inequality $\log(a)> 1 - a^{-1}$ for any $a>0$, we obtain
	\[
	\alpha x_{\varepsilon}\log\left(\varepsilon\right)+k\left(1-x_{\varepsilon}\right)\log\left(\log\left(\varepsilon^{-\beta}\right)\right)=\log\left(x_{\varepsilon}\right)>1-\frac{1}{x_{\varepsilon}},
	\]
	or equivalently,
	\begin{align}\label{logineq}
	\left[\alpha\log\left(\varepsilon\right)-k\log\left(\log\left(\varepsilon^{-\beta}\right)\right)\right]x_{\varepsilon}^{2}-\left(1-k\log\left(\log\left(\varepsilon^{-\beta}\right)\right)\right)x_{\varepsilon}+1>0.
	\end{align}
	In view of the facts that $\alpha\log\left(\varepsilon\right)-k\log\left(\log\left(\varepsilon^{-\beta}\right)\right)<0$ and 
	\begin{align*}
		D_{\varepsilon} := \left(1-k\log\left(\log\left(\varepsilon^{-\beta}\right)\right)\right)^{2} & -4\left[\alpha\log\left(\varepsilon\right)-k\log\left(\log\left(\varepsilon^{-\beta}\right)\right)\right]\\
		& =\left(1-k\log\left(\log\left(\varepsilon^{-\beta}\right)\right)\right)^{2}+4k\log\left(\log\left(\varepsilon^{-\beta}\right)\right)+4\log\left(\varepsilon^{-\alpha}\right)>0,
	\end{align*}
	the above quadratic inequality \eqref{logineq} admits the following solution:
	\[
	x_{\varepsilon}\in\left(\frac{1-k\log\left(\log\left(\varepsilon^{-\beta}\right)\right)+\sqrt{D_{\varepsilon}}}{2\left(\alpha\log\left(\varepsilon\right)-k\log\left(\log\left(\varepsilon^{-\beta}\right)\right)\right)},\frac{1-k\log\left(\log\left(\varepsilon^{-\beta}\right)\right)-\sqrt{D_{\varepsilon}}}{2\left(\alpha\log\left(\varepsilon\right)-k\log\left(\log\left(\varepsilon^{-\beta}\right)\right)\right)}\right).
	\]
	By the rationalizing technique, we show the zero limit of $x_{\varepsilon}$ as follows:
	\begin{align*}
		& \lim_{\varepsilon\to0}\left[\frac{1-k\log\left(\log\left(\varepsilon^{-\beta}\right)\right)-\sqrt{D_{\varepsilon}}}{2\left(\alpha\log\left(\varepsilon\right)-k\log\left(\log\left(\varepsilon^{-\beta}\right)\right)\right)}\right]\\
		& =\lim_{\varepsilon\to0}\frac{\left[1-k\log\left(\log\left(\varepsilon^{-\beta}\right)\right)\right]^{2}-D_{\varepsilon}}{2\left(\alpha\log\left(\varepsilon\right)-k\log\left(\log\left(\varepsilon^{-\beta}\right)\right)\right)\left(1-k\log\left(\log\left(\varepsilon^{-\beta}\right)\right)+\sqrt{D_{\varepsilon}}\right)}\\
		& =\lim_{\varepsilon\to0}\frac{2}{1-k\log\left(\log\left(\varepsilon^{-\beta}\right)\right)+\sqrt{D_{\varepsilon}}}=0.
	\end{align*}
	Hence, it follows from \eqref{uni} that
	\[
	\left\Vert u^{\varepsilon}\left(1-x_{\varepsilon},\cdot\right)-u\left(1,\cdot\right)\right\Vert \le\frac{2C+2M}{1-k\log\left(\log\left(\varepsilon^{-\beta}\right)\right)+\sqrt{D_{\varepsilon}}}.
	\]
	Therefore, we complete the proof of the theorem.
\end{proof}

%

\section{Iterative scheme}\label{sec:4}

In the previous section, we have studied the strong convergence of $u^{\varepsilon}$ toward the exact solution $u$. Observe that by the choice of the perturbation and stabilization in our regularization problem \eqref{reguHelm} is not really computable, albeit the problem is linear in terms of its solution and the series is truncated appropriately in $\varepsilon$. For each $\varepsilon>0$, we construct an iterative sequence $\left\{u^{\varepsilon,q}\right\}_{q\in\mathbb{N}}$ to approximate $u^{\varepsilon}$ of \eqref{reguHelm} in the following sense.
\begin{align}\label{iterative}
\begin{cases}
	u_{xx}^{\varepsilon,q+1}-u_{yy}^{\varepsilon,q+1}+\mathbf{P}_{\varepsilon}u^{\varepsilon,q}=k^{2}u^{\varepsilon,q} & \text{in }\Omega,\\
	u^{\varepsilon,q+1}\left(x,0\right)=u^{\varepsilon,q+1}\left(x,1\right)=0 & \text{for }x\in\left(0,1\right),\\
	u^{\varepsilon,q+1}\left(0,y\right)=u_{0}^{\varepsilon}\left(y\right),\quad u_{x}^{\varepsilon,q+1}\left(0,y\right)=0 & \text{for }y\in\left(0,1\right).
\end{cases}
\end{align}
In the above iteration scheme, we choose the initial guess $u^{\varepsilon,0}$ (i.e. $q=0$) is chosen to be $u_0^{\varepsilon}(y)$. We choose this initial guess because it is a unique function that contains close information of our sought $u^{\varepsilon}$ under stabilization. Even though proposing this iterative scheme can be a curse of dimensionality, our previous work \cite{Khoa2020a} shows numerically that we only need a very small amount of iteration steps (about $q=2$) to obtain a fine approximation.

For every $\varepsilon$, we can divide the interval $[0,1]$ of $x$ into many finite subintervals. Our convergence result below shows that the mesh-width in $x$ should be dependent of the noise level $\varepsilon$ for local approximation of the regularized solution $u^{\varepsilon}$. It is sufficient to analyze the convergence of the linearization in a subinterval $[0,\overline{x}]\subset [0,1]$ since we can repeat the linearization procedure in every subinterval. Below, we prove the strong convergence of the scheme in a suitable topology involving the space  $\Upsilon_{\overline{x}} = C\left(\left[0,\overline{x}\right];L^{2}\left(0,1\right)\right)$. 

\begin{theorem}[Convergence of linearization]
	Under the assumptions of Theorem \ref{thm:4-1}, the approximate solution $u^{\varepsilon,q}$ defined in \eqref{iterative} is strongly convergent in $\Upsilon_{\Delta x}$. Moreover, for each $\varepsilon>0$, there exists a sufficiently small $\eta_{\varepsilon}\in (0,1)$ such that for $\sigma \ge 1$,
	\begin{align*}
		& \left\Vert u_{y}^{\varepsilon,q}-u_{y}^{\varepsilon}\right\Vert _{\Upsilon_{\overline{x}}}+\sigma\log\left(\varepsilon^{-\alpha}\right)\left\Vert u^{\varepsilon,q}-u^{\varepsilon}\right\Vert _{\Upsilon_{\overline{x}}}\le\frac{\eta_{\varepsilon}^{q}}{1-\eta_{\varepsilon}}\left(\left\Vert u_{y}^{\varepsilon,1}-u_{y}^{\varepsilon,0}\right\Vert _{\Upsilon_{\overline{x}}}+k\left\Vert u^{\varepsilon,1}-u^{\varepsilon,0}\right\Vert _{\Upsilon_{\overline{x}}}\right),\\
		& \left\Vert u_{x}^{\varepsilon,q}-u_{x}^{\varepsilon}\right\Vert _{\Upsilon_{\overline{x}}}\le\frac{\eta_{\varepsilon}^{q}}{1-\eta_{\varepsilon}}\left(\left\Vert u_{y}^{\varepsilon,1}-u_{y}^{\varepsilon,0}\right\Vert _{\Upsilon_{\overline{x}}}+k\left\Vert u^{\varepsilon,1}-u^{\varepsilon,0}\right\Vert _{\Upsilon_{\overline{x}}}\right).
	\end{align*}
\end{theorem}
\begin{proof}
	We proceed as in the proof of Theorem \ref{thm:4-1}, scrutinizing energy estimates under the Carleman weight of the form $e^{-\kappa x}$. Let $ W^{q+1} = e^{-\kappa x}\left(u^{\varepsilon,q+1} - u^{\varepsilon,q}\right)$ where $\kappa > 0$ is a constant chosen later. Thus, $W^{q+1}$ satisfies the following system:
	\[
	\begin{cases}
		W_{xx}^{q+1}-W_{yy}^{q+1}+2\kappa W_{x}^{q+1} + \kappa^2 W^{q+1}=-\mathbf{P}_{\varepsilon}W^{q} + k^2 W^{q}& \text{in }\Omega\\
		W^{q+1}\left(x,0\right)=W^{q+1}\left(x,1\right)=0 & \text{for }x\in\left(0,1\right),\\
		W^{q+1}\left(0,y\right)=0,\quad W_{x}^{q+1}\left(0,y\right)=0 & \text{for }y\in\left(0,1\right).
	\end{cases}
	\]
	Multiplying the difference equation by $W_{x}^{k+1}$ and then integrating the resulting equation from 0 to 1, we find that
	\begin{align*}
		\frac{d}{dx}\left\Vert W_{x}^{q+1}\left(x,\cdot\right)\right\Vert ^{2} & +\frac{d}{dx}\left\Vert W_{y}^{q+1}\left(x,\cdot\right)\right\Vert ^{2}+\kappa^{2}\frac{d}{dx}\left\Vert W^{q+1}\left(x,\cdot\right)\right\Vert ^{2}\\
		& =-2\left\langle \mathbf{P}_{\varepsilon}W^{q},W_{x}^{q+1}\right\rangle +2k^{2}\left\langle W^{q},W_{x}^{q+1}\right\rangle -4\kappa\left\Vert W_{x}^{q+1}\left(x,\cdot\right)\right\Vert ^{2}.
	\end{align*}
	Integrating the above equation from 0 to $x$ and choosing $\kappa = \sigma\log(\gamma) \ge k$ for $\sigma\ge 1$, we estimate that
	\begin{align}
		 & \left\Vert W_{x}^{q+1}\left(x,\cdot\right)\right\Vert ^{2}+\left\Vert W_{y}^{q+1}\left(x,\cdot\right)\right\Vert ^{2}+\kappa^{2}\left\Vert W^{q+1}\left(x,\cdot\right)\right\Vert ^{2} \nonumber \\
		& \le\int_{0}^{x}\left[\log\left(\gamma\right)\left\Vert W_{y}^{q}\left(s,\cdot\right)\right\Vert ^{2}+k^{2}\left\Vert W^{q}\left(s,\cdot\right)\right\Vert ^{2}+\left(4\log\left(\gamma\right)+1\right)\left\Vert W_{x}^{q+1}\left(s,\cdot\right)\right\Vert ^{2}-4\kappa\left\Vert W_{x}^{q+1}\left(s,\cdot\right)\right\Vert ^{2}\right]ds\nonumber\\
		& \le\overline{x}\left(\log\left(\gamma\right)\left\Vert W_{y}^{q}\right\Vert _{\Upsilon_{\overline{x}}}^{2}+k^{2}\left\Vert W^{q}\right\Vert _{\Upsilon_{\overline{x}}}^{2}\right)+\int_{0}^{x}\left\Vert W_{x}^{q+1}\left(s,\cdot\right)\right\Vert ^{2}ds.\nonumber 
	\end{align}
	By the Gronwall inequality, we have
	\begin{align}\label{West}
	\left\Vert W_{x}^{q+1}\left(x,\cdot\right)\right\Vert ^{2}+\left\Vert W_{y}^{q+1}\left(x,\cdot\right)\right\Vert ^{2}+\kappa^{2}\left\Vert W^{q+1}\left(x,\cdot\right)\right\Vert ^{2}\le\overline{x}\left(\log\left(\gamma\right)\left\Vert W_{y}^{q}\right\Vert _{\Upsilon_{\overline{x}}}^{2}+k^{2}\left\Vert W^{q}\right\Vert _{\Upsilon_{\overline{x}}}^{2}\right)e^{\overline{x}}.
	\end{align}
	Dropping the first term on the left-hand side of \eqref{West}, we, after back-substitution, get
	\begin{align*}
		\left\Vert u_{y}^{\varepsilon,q+1}\left(x,\cdot\right)-u_{y}^{\varepsilon,q}\left(x,\cdot\right)\right\Vert ^{2} & +\kappa^{2}\left\Vert u^{\varepsilon,q+1}\left(x,\cdot\right)-u^{\varepsilon,q}\left(x,\cdot\right)\right\Vert ^{2}\\
		& \le\overline{x}\log\left(\gamma\right)e^{\overline{x}}\left(\left\Vert u_{y}^{\varepsilon,q}-u_{y}^{\varepsilon,q-1}\right\Vert _{\Upsilon_{\overline{x}}}^{2}+k\left\Vert u^{\varepsilon,q}-u^{\varepsilon,q-1}\right\Vert _{\Upsilon_{\overline{x}}}^{2}\right)e^{2\kappa x}.
	\end{align*}
	This leads to
	\begin{align}
	\left\Vert u_{y}^{\varepsilon,q+1}-u_{y}^{\varepsilon,q}\right\Vert _{\Upsilon_{\overline{x}}}^{2}+\kappa^{2}\left\Vert u^{\varepsilon,q+1}-u^{\varepsilon,q}\right\Vert _{\Upsilon_{\overline{x}}}^{2}\le\overline{x}\log\left(\gamma\right)e^{\overline{x}}\gamma^{2\sigma\overline{x}}\left(\left\Vert u_{y}^{\varepsilon,q}-u_{y}^{\varepsilon,q-1}\right\Vert _{\Upsilon_{\overline{x}}}^{2}+k\left\Vert u^{\varepsilon,q}-u^{\varepsilon,q-1}\right\Vert _{\Upsilon_{\overline{x}}}^{2}\right).\label{West2}
	\end{align}
	Next, by the standard inequality $(a-b)^2 \ge \frac{1}{2}a^2 - b^2$ for all $a,b\in\mathbb{R}$, we have
	\begin{align}
		& \frac{1}{2}\left\Vert u_{x}^{\varepsilon,q+1}\left(x,\cdot\right)-u_{x}^{\varepsilon,q}\left(x,\cdot\right)\right\Vert ^{2}-\kappa^{2}\left\Vert u^{\varepsilon,q+1}\left(x,\cdot\right)-u^{\varepsilon,q}\left(x,\cdot\right)\right\Vert ^{2}\nonumber \\
		& \le\left\Vert u_{x}^{\varepsilon,q+1}\left(x,\cdot\right)-u_{x}^{\varepsilon,q}\left(x,\cdot\right)-\kappa\left[u^{\varepsilon,q+1}\left(x,\cdot\right)-u^{\varepsilon,q}\left(x,\cdot\right)\right]\right\Vert ^{2}=e^{2\kappa x}\left\Vert W_{x}^{q+1}\left(x,\cdot\right)\right\Vert ^{2}\nonumber \\
		& \le\gamma^{2\sigma x}\overline{x}\left(\log\left(\gamma\right)\left\Vert u_{y}^{\varepsilon,q}-u_{y}^{\varepsilon,q-1}\right\Vert _{\Upsilon_{\overline{x}}}^{2}+k^{2}\left\Vert u^{\varepsilon,q}-u^{\varepsilon,q-1}\right\Vert _{\Upsilon_{\overline{x}}}^{2}\right)e^{\overline{x}}.\label{West1}
	\end{align}
	where we have used dropping the second and third terms on the left-hand side of \eqref{West}. It now follows from \eqref{West1} and \eqref{West} that
	\begin{align}
		\left\Vert u_{x}^{\varepsilon,q+1}-u_{x}^{\varepsilon,q}\right\Vert _{\Upsilon_{\overline{x}}}^{2} & \le2\kappa^{2}\left\Vert u^{\varepsilon,q+1}-u^{\varepsilon,q}\right\Vert _{\Upsilon_{\overline{x}}}^{2}+2\gamma^{2\sigma\overline{x}}\overline{x}\left(\log\left(\gamma\right)\left\Vert u_{y}^{\varepsilon,q}-u_{y}^{\varepsilon,q-1}\right\Vert _{\Upsilon_{\overline{x}}}^{2}+k^{2}\left\Vert u^{\varepsilon,q}-u^{\varepsilon,q-1}\right\Vert _{\Upsilon_{\overline{x}}}^{2}\right)e^{\overline{x}} \nonumber \\
		& \le4\overline{x}e^{\overline{x}}\gamma^{2\sigma\overline{x}}\left(\log\left(\gamma\right)\left\Vert u_{y}^{\varepsilon,q}-u_{y}^{\varepsilon,q-1}\right\Vert _{\Upsilon_{\overline{x}}}^{2}+k^{2}\left\Vert u^{\varepsilon,q}-u^{\varepsilon,q-1}\right\Vert _{\Upsilon_{\overline{x}}}^{2}\right)\nonumber \\
		& \le4\overline{x}e^{\overline{x}}\gamma^{2\sigma\overline{x}}\log\left(\gamma\right)\left(\left\Vert u_{y}^{\varepsilon,q}-u_{y}^{\varepsilon,q-1}\right\Vert _{\Upsilon_{\overline{x}}}^{2}+k\left\Vert u^{\varepsilon,q}-u^{\varepsilon,q-1}\right\Vert _{\Upsilon_{\overline{x}}}^{2}\right).
		\label{West3}
	\end{align}
	Therefore, we choose $\overline{x}$ small enough such that
	\begin{align}\label{eta}
	\eta_{\gamma}^{2}:=4\overline{x}e^{\overline{x}}\gamma^{2\sigma\overline{x}}\log\left(\gamma\right)<1.
	\end{align}
	Then, using the Minkowski inequality $(a+b)^2 \le 2(a^2 + b^2)$ for $a,b\in\mathbb{R}$ and triangle inequality, we deduce from \eqref{West2} that for $l\ge 1$,
	\begin{align*}
		&\left\Vert u_{y}^{\varepsilon,q+l}-u_{y}^{\varepsilon,q}\right\Vert _{\Upsilon_{\overline{x}}}  +\kappa\left\Vert u^{\varepsilon,q+l}-u^{\varepsilon,q}\right\Vert _{\Upsilon_{\overline{x}}}\\
		& \le\sum_{j=1}^{l}\left(\left\Vert u_{y}^{\varepsilon,q+j}-u_{y}^{\varepsilon,q+j-1}\right\Vert _{\Upsilon_{\overline{x}}}+\kappa\left\Vert u^{\varepsilon,q+j}-u^{\varepsilon,q+j-1}\right\Vert _{\Upsilon_{\overline{x}}}\right)\\
		& \le\sum_{j=1}^{l}\eta_{\gamma}^{q+j-1}\left(\left\Vert u_{y}^{\varepsilon,1}-u_{y}^{\varepsilon,0}\right\Vert _{\Upsilon_{\overline{x}}} + k\left\Vert u^{\varepsilon,1}-u^{\varepsilon,0}\right\Vert _{\Upsilon_{\overline{x}}} \right)=\frac{\eta_{\gamma}^{q}\left(1-\eta_{\gamma}^{l}\right)}{1-\eta_{\gamma}}\left(\left\Vert u_{y}^{\varepsilon,1}-u_{y}^{\varepsilon,0}\right\Vert _{\Upsilon_{\overline{x}}}+k\left\Vert u^{\varepsilon,1}-u^{\varepsilon,0}\right\Vert _{\Upsilon_{\overline{x}}}\right).
	\end{align*}
	Henceforth, $\left\{u^{\varepsilon,q}\right\}_{q\in\mathbb{N}}$ and $\left\{u_y^{\varepsilon,q}\right\}_{q\in\mathbb{N}}$ are Cauchy sequences in $\Upsilon_{\overline{x}}$, respectively. Thus, there exists uniquely $u^{\varepsilon}\in \Upsilon_{\overline{x}}$ such that $u^{\varepsilon,q}\to u^{\varepsilon}$ strongly in $\Upsilon_{\overline{x}}$ as $q\to \infty$. Similarly, we obtain a unique $u_{y}^{\varepsilon}\in \Upsilon_{\overline{x}}$ such that $u_{y}^{\varepsilon,q}\to u_{y}^{\varepsilon}$ strongly in $\Upsilon_{\overline{x}}$ as $q\to \infty$. By \eqref{West3}, we also obtain that $\left\{u_x^{\varepsilon,q}\right\}_{q\in\mathbb{N}}$ is a Cauchy sequence in $\Upsilon_{\overline{x}}$ and thus, there exists a unique limit $u_x^{\varepsilon}$ that converges strongly to $u_x^{\varepsilon,q}$ in $\Upsilon_{\overline{x}}$. Moreover, taking $l\to\infty $ we have
	\begin{align*}
		& \left\Vert u_{y}^{\varepsilon,q}-u_{y}^{\varepsilon}\right\Vert _{\Upsilon_{\overline{x}}}+\kappa\left\Vert u^{\varepsilon,q}-u^{\varepsilon}\right\Vert _{\Upsilon_{\overline{x}}}\le\frac{\eta_{\gamma}^{q}}{1-\eta_{\gamma}}\left(\left\Vert u_{y}^{\varepsilon,1}-u_{y}^{\varepsilon,0}\right\Vert _{\Upsilon_{\overline{x}}}+k\left\Vert u^{\varepsilon,1}-u^{\varepsilon,0}\right\Vert _{\Upsilon_{\overline{x}}}\right),\\
		& \left\Vert u_{x}^{\varepsilon,q}-u_{x}^{\varepsilon}\right\Vert _{\Upsilon_{\overline{x}}}\le\frac{\eta_{\gamma}^{q}}{1-\eta_{\gamma}}\left(\left\Vert u_{y}^{\varepsilon,1}-u_{y}^{\varepsilon,0}\right\Vert _{\Upsilon_{\overline{x}}}+k\left\Vert u^{\varepsilon,1}-u^{\varepsilon,0}\right\Vert _{\Upsilon_{\overline{x}}}\right).
	\end{align*}
	In addition, we obtain the strong convergence (as $q\to \infty$) $\mathbf{P}_{\varepsilon}u^{\varepsilon,q}$ in the following manner:
	\[
	\left\Vert \mathbf{P}_{\varepsilon}u^{\varepsilon,q}-\mathbf{P}_{\varepsilon}u^{\varepsilon}\right\Vert _{\Upsilon_{\overline{x}}}\le2\log\left(\gamma\right)\left\Vert u_{y}^{\varepsilon,q}-u_{y}^{\varepsilon}\right\Vert _{\Upsilon_{\overline{x}}}\le\frac{2\log\left(\gamma\right)\eta_{\gamma}^{q}}{1-\eta_{\gamma}}\left(\left\Vert u_{y}^{\varepsilon,1}-u_{y}^{\varepsilon,0}\right\Vert _{\Upsilon_{\overline{x}}}+k\left\Vert u^{\varepsilon,1}-u^{\varepsilon,0}\right\Vert _{\Upsilon_{\overline{x}}}\right).
	\]
	Hence, the limit $u^{\varepsilon}\in \Upsilon_{\overline{x}}$ found above is the solution of the regularized system \eqref{reguHelm} in the subinterval $[0,\overline{x}]$. We complete the proof of the theorem.
\end{proof}


\begin{remark}
	It is not hard to see that we do not really need to linearize the term $k^2 u$ on the right-hand side of \eqref{reguHelm}, while the convergence is still guaranteed from the theoretical standpoint.  However, numerical observations show that linearization of the term $k^2 u$ give better numerical results. This mainly explains why we choose the current linearization procedure.
\end{remark}

\section{Numerical examples}\label{sec:num}

\subsection{Finite difference settings}

Given $M,N\in\mathbb{N}$, we consider uniform grids of mesh-points $x_m = (m-1)\Delta x, y_n = (n-1)\Delta y$ for $1\le m\leq M+1, 1\le n\le N+1$ with  $\Delta x,\Delta y$ being the mesh-widths in $x$ and $y$, respectively. For any function $u(x,y)$, we denote by $u_{m,n}\approx u(x_m,y_n)$ the corresponding discrete function. To generate the data, we apply the central finite difference method (FDM) to solve the Helmholtz equation \eqref{helmholtz} with the Dirichlet boundary conditions imposed on four sides of $\Omega=(0,1)^2$, viz.
\begin{align}\label{boundHelm}
	u(0,y) = u_0(y),\quad u(1,y)= g(y),\quad u(x,0)=0,\quad u(x,1)=0.
\end{align}
In our numerical performance of the stabilization scheme below, we do not choose the true solution of the Helmholtz equation \eqref{helmholtz}. Instead, we choose its boundary data $u_0,g$ in \eqref{boundHelm} so that our choice is more flexible. This is relevant because  \eqref{helmholtz} with full data \eqref{boundHelm} is a well-posed problem and the central FDM is well known to be stable and convergent with respect to the refinement of $x$ and $y$. In this circumstance, one can consider the discrete function $u_{m,n}$ obtained from that well-posed problem as a reliable true solution. The Neumann data $u_1$ in \eqref{boundary1} can be generated using the fact that
\[
u_0(y_n)\approx u_1(y_n) \Delta x + u(x_2,y_n). 
\]
 
The same FDM is applied when we solve $U(x,y)$ of system \eqref{UU}. For ease of presentation, we only detail below this FDM for $U(x,y)$, while the scheme for $u$ can be established in the same manner. The center approximation for partial derivatives with respect to $x$ and $y$ is given by
\begin{align}\label{diffop}
	U_{xx}(x_m,y_n) \approx \dfrac{U_{m+1,n}-2U_{m,n}+U_{m-1,n}}{(\Delta x)^2},\quad 
	U_{yy}(x_m,y_n) \approx \dfrac{U_{m,n+1}-2U_{m,n}+U_{m,n-1}}{(\Delta y)^2},
\end{align}
Thus, the PDE in \eqref{UU} is discretized as follows:
\begin{align*}
	\dfrac{U_{m+1,n}-2U_{m,n}+U_{m-1,n}}{(\Delta x)^2}+\dfrac{U_{m,n+1}-2U_{m,n}+U_{m,n-1}}{(\Delta y)^2}+k^2U_{m,n} = 0.
\end{align*}
Put $r=\Delta x/\Delta y$. We obtain
\begin{align}\label{matrix1}
	U_{m+1,n}+U_{m-1,n}+r^2U_{m,n+1}+\left[\left(k\Delta x\right)^2-2-2r^2\right]U_{m,n}+r^2U_{m,n-1} &= 0.
\end{align}
Denote the unknown $\textbf{U}_m = \left(U_{m,2},U_{m,3},U_{m,4},\ldots,U_{m,N}\right)^\text{T}$. We rewrite \eqref{matrix1} in the following matrix form:
\begin{align*}
	\begin{bmatrix}
		\mathcal{K}_1 & \textbf{I}_{N-1} & \textbf{0} & \ldots & \textbf{0}\\
		\textbf{I}_{N-1} & \mathcal{K}_{2} & \textbf{I}_{N-1} & \ldots & \textbf{0}\\
		\textbf{0} & \textbf{I}_{N-1} & \mathcal{K}_{2} & \ldots & \textbf{0}\\
		\vdots & \vdots & \vdots & \ddots & \vdots\\
		\textbf{0} & \textbf{0} & \textbf{0} & \ldots & \mathcal{K}_{2}
	\end{bmatrix}
	\begin{bmatrix}
		\textbf{U}_2\\
		\textbf{U}_3\\
		\textbf{U}_4\\
		\vdots\\
		\textbf{U}_M
	\end{bmatrix} = \begin{bmatrix}
		\textbf{F}_2\\
		\textbf{F}_3\\
		\textbf{F}_4\\
		\vdots\\
		\textbf{F}_M
	\end{bmatrix},
\end{align*}
where $\mathbf{I}_{N-1}\in \mathbb{M}^{(N-1)\times(N-1)}$ stands for the identity matrix, the block matrices $\mathcal{K}_{1},\mathcal{K}_{2}\in\mathbb{M}^{(N-1)\times(N-1)}$ are defined with $T_k = \left(k\Delta x\right)^2-2-2r^2$, as follows:
\begin{align*}
	\mathcal{K}_1 = \begin{bmatrix}
		T_k+1 & r^2 & 0 & \ldots & 0\\
		r^2 & T_k+1 & r^2 & \ldots & 0\\
		0 & r^2 & T_k+1 & \ldots & 0\\
		\vdots & \vdots & \vdots & \ddots & \vdots\\
		0 & 0 & 0 & \ldots & T_k+1
	\end{bmatrix},\mathcal{K}_{2} = \begin{bmatrix}
		T_k & r^2 & 0 & \ldots & 0\\
		r^2 & T_k & r^2 & \ldots & 0\\
		0 & r^2 & T_k & \ldots & 0\\
		\vdots & \vdots & \vdots & \ddots & \vdots\\
		0 & 0 & 0 & \ldots & T_k
	\end{bmatrix},
\end{align*}
and vectors $\textbf{F}_m$ are denoted by
\begin{align*}
	\textbf{F}_2 = \Delta x\left(u_1(y_2),u_1(y_3),u_1(y_4),\ldots,u_1(y_{N})\right)^\text{T},\quad \textbf{F}_m = \left(0,0,0,\ldots,0\right)^\text{T},\quad 3\leq m\leq M.
\end{align*}

Solving \eqref{matrix1} allows us to find a numerical solution of $U(x,y)$ to system \eqref{UU}. Thereby, it follows that an approximation of $U(0,y)$ can be obtained for the Dirichlet data in \eqref{VV}. To solve for $V(x,y)$ in \eqref{VV} numerically, we accordingly apply the iterative scheme investigated in section \ref{sec:4}. That means we construct a sequence of $\left\{V^{\varepsilon,q}\right\}_{q\in\mathbb{N}}$ satisfying
\begin{align}
\begin{cases}\label{systemV}
	V_{xx}^{\varepsilon,q+1}-V_{yy}^{\varepsilon,q+1}+\mathbf{P}V^{\varepsilon,q}-k^{2}V^{\varepsilon,q}=0 & \text{in }\Omega,\\
	V^{\varepsilon,q+1}\left(x,0\right)=V^{\varepsilon,q+1}(x,1)=0 & \text{for }x\in\left(0,1\right),\\
	V^{\varepsilon,q+1}\left(0,y\right)=u_{0}^{\varepsilon}\left(y\right)-U\left(0,y\right),\quad V_{x}^{\varepsilon,q+1}\left(0,y\right)=0 & \text{for }y\in\left(0,1\right).
\end{cases}
\end{align}
In \eqref{systemV}, we recall that
\begin{align*}
	\mathbf{P}u(x,\cdot) = -2\sum_{j\in\mathbb{N}\backslash (B\cup A_3)}\lambda_{j,k}\left\langle u(x,\cdot),\phi_{j}\right\rangle \phi_{j} = -2\sum_{j\in\mathbb{N}\backslash(B\cup A_3)}\left(\mu_j-k^2\right)\left\langle u(x,\cdot),\phi_j\right\rangle\phi_j,
\end{align*}
and the initial guess $V^{\varepsilon,0}$ (i.e. $q=0$) is chosen to be $u_0^{\varepsilon}(y) - U(0,y)$. As mentioned in section \ref{sec:4}, we choose this initial guess because it is a unique function that contains many information of our sought $V^{\varepsilon}$ under stabilization. Let $V^{\varepsilon,q}(x_i,y_j) \approx V^{\varepsilon,q}_{i,j}$, and the same difference operators in \eqref{diffop} are applied to the PDE of \eqref{systemV}. It yields that
\begin{align*}
	V_{xx}^{\varepsilon,q+1}(x_m,y_n) \approx \dfrac{V_{m+1,n}^{\varepsilon,q+1}-2V_{m,n}^{\varepsilon,q+1}+V_{m-1,n}^{\varepsilon,q+1}}{(\Delta x)^2},\quad
	V_{yy}^{\varepsilon,q+1}(x_m,y_n) \approx \dfrac{V_{m,n+1}^{\varepsilon,q+1}-2V_{m,n}^{\varepsilon,q+1}+V_{m,n-1}^{\varepsilon,q+1}}{(\Delta y)^2}.
\end{align*}
Combining these with the standard Riemann sum approximating the inner product in $\mathbf{P}$, we seek $V_{m,n}^{\varepsilon,q}$ satisfying the following approximate equation:
\begin{align*}
	\dfrac{V_{m+1,n}^{\varepsilon,q+1}-2V_{m,n}^{\varepsilon,q+1}+V_{m-1,n}^{\varepsilon,q+1}}{(\Delta x)^2}&-\dfrac{V_{m,n+1}^{\varepsilon,q+1}-2V_{m,n}^{\varepsilon,q+1}+V_{m,n-1}^{\varepsilon,q+1}}{(\Delta y)^2}\\
	&-2\Delta y\sum_{j\in\mathbb{N}\backslash(B\cup A_3)}\left(\mu_j-k^2\right)\sum_{l=1}^{N+1}V^{\varepsilon,q}_{m,l}\phi_j(y_l)\phi_j(y_n) = k^2V^{\varepsilon,q}_{m,n}.
\end{align*}
Recall that $r=\Delta x/ \Delta y$. We get
\begin{align*}
	V_{m+1,n}^{\varepsilon,q+1} &= r^2V_{m,n+1}^{\varepsilon,q+1}+\left(2-2r^2\right)V_{m,n}^{\varepsilon,q+1}+r^2V_{m,n-1}^{\varepsilon,q+1}-V_{m-1,n}^{\varepsilon,q+1}\\
	&+2(\Delta x)^2\Delta y\sum_{j\in\mathbb{N}\backslash(B\cup A_3)}\left(\mu_j-k^2\right)\sum_{l=1}^{N+1}V^{\varepsilon,q}_{m,l}\phi_j(y_l)\phi_j(y_n)+k^2(\Delta x)^2V_{m,n}^{\varepsilon,q}.
\end{align*}
Let $\textbf{V}_m^{\varepsilon,q}=\left(V_{m,2}^{\varepsilon,q},V_{m,3}^{\varepsilon,q},V_{m,4}^{\varepsilon,q},\ldots,V_{m,N}^{\varepsilon,q}\right)^\text{T}$. The above equation can be rewritten in the following matrix form:
\begin{align}\label{matrix2}
	\textbf{V}_{m+1}^{\varepsilon,q+1} = \textbf{K}\textbf{V}_{m}^{\varepsilon,q+1}-\textbf{V}_{m-1}^{\varepsilon,q+1}+\textbf{f}(\textbf{V}_m^{\varepsilon,q})+k^2(\Delta x)^2\textbf{V}_m^{\varepsilon,q},
\end{align}
where we have denoted by
\begin{align*}
	\textbf{K}=\begin{bmatrix}
		2-2r^2 & r^2 & 0 & \ldots & 0\\
		r^2 & 2-2r^2 & r^2 & \ldots & 0\\
		0 & r^2 & 2-2r^2 & \ldots & 0\\
		\vdots & \vdots & \vdots & \ddots & \vdots\\
		0 & 0 & 0 & \ldots & 2-2r^2
	\end{bmatrix},\quad
	\textbf{f}(\textbf{V}_m^{\varepsilon,q}) = \begin{bmatrix}
		\textbf{f}\left(\textbf{V}_m^{\varepsilon,q}\right)\left(y_2\right)\\
		\textbf{f}\left(\textbf{V}_m^{\varepsilon,q}\right)\left(y_3\right)\\
		\textbf{f}\left(\textbf{V}_m^{\varepsilon,q}\right)\left(y_4\right)\\
		\vdots\\
		\textbf{f}\left(\textbf{V}_m^{\varepsilon,q}\right)\left(y_N\right)
	\end{bmatrix}.
\end{align*}
Herewith, elements in $\textbf{f}(\textbf{V}_m^{\varepsilon,q})$ are understood as
\begin{align*}
	\textbf{f}\left(\textbf{V}_m^{\varepsilon,q}\right)\left(y_n\right) &= 2(\Delta x)^2\Delta y\sum_{\frac{k^2}{\pi^2}\leq j^2\leq \frac{k^2+\log^2(\gamma)}{\pi^2}}\left(\mu_j-k^2\right)\sum_{l=1}^{N+1}V^{\varepsilon,q}_{m,l}\phi_j(y_l)\phi_j(y_n)\\
	&= 2(\Delta x)^2\Delta y\sum_{\frac{k^2}{\pi^2}\leq j^2\leq \frac{k^2+\log^2(\gamma)}{\pi^2}}\left(\mu_j-k^2\right)\begin{bmatrix}
		\phi_j(y_2)\\
		\phi_j(y_3)\\
		\phi_j(y_4)\\
		\vdots\\
		\phi_j(y_N)
	\end{bmatrix}^\text{T}\begin{bmatrix}
		V_{m,2}^{\varepsilon,q}\\
		V_{m,3}^{\varepsilon,q}\\
		V_{m,4}^{\varepsilon,q}\\
		\vdots\\
		V_{m,N}^{\varepsilon,q}\\
	\end{bmatrix}\phi_j(y_n).
\end{align*}

After having $V^{\varepsilon,q}$ from \eqref{matrix2}, we obtain an approximation of $u^{\varepsilon}$ via $u^{\varepsilon}_{m,n} = U_{m,n} + V^{\varepsilon,q}_{m,n}$. As to the measured data $u_{0}^{\varepsilon}$ in \eqref{systemV}, we apply the additive noise in the following sense: $u_{0}^{\varepsilon}(y) = u_0(y) + \varepsilon\text{rand}(y)$, where rand is a uniformly distributed random number such that $\max_{y\in[0,1]}\left|\text{rand}(y)\right|\le 1/(2N)$. At the discretization level, the gradient of $u_{0}^{\varepsilon}$ is then approximated by
\[
\partial_y u_{0}^{\varepsilon}(y_n) \approx \frac{u_0^{\varepsilon}(y_{n+1})-u_0^{\varepsilon}(y_{n})}{\Delta y} \approx \partial_y u_{0} (y_n) + \varepsilon N \left(\text{rand}(y_{n+1})-\text{rand}(y_{n})\right).
\]
Therefore, we can see that assumption \eqref{noisecondition} is fulfilled. The (local) convergence of the linerization scheme for \eqref{systemV} has been studied in section \ref{sec:4}. In this regard, we, according to \eqref{eta}, condition that
\begin{align}\label{eta1}
	\eta_{\varepsilon}^2=4\Delta xe^{\Delta x}\varepsilon^{-2\alpha\Delta x}\log\left(\varepsilon^{-\alpha}\right)< 1
\end{align}
indicating $\sigma = 1$ is taken. Henceforth, a suitable fine mesh for variable $x$ should be applied. Below, we fix $\alpha = 1$ and $N=40$ when enjoying the numerical performance of the QR scheme for different noise levels. Moreover, we choose $q = 1$ in our iterative procedure for the QR scheme. The choice of $M$ will be specified in each example since cf. \eqref{eta1}, it depends on values of $\varepsilon$. Also, for simplicity, we take $g(y)=0$ for all examples below, while varying $u_0(y)$ in \eqref{boundHelm}. Last but not least, we below consider the following relative error:
\[
E=\frac{\sqrt{\sum_{m=0}^{M}\sum_{n=0}^{N}\left|u_{m,n}^{\varepsilon}-u_{m,n}\right|^{2}}}{\sqrt{\sum_{m=0}^{M}\sum_{n=0}^{N}\left|u_{m,n}\right|^{2}}}\times100\%.
\]

\subsection{Numerical performance for variable noise levels}

\subsection*{Example 1: Low frequency} 

\begin{figure}
	\subfloat[True\label{fig:1a}]{\includegraphics[scale=0.3]{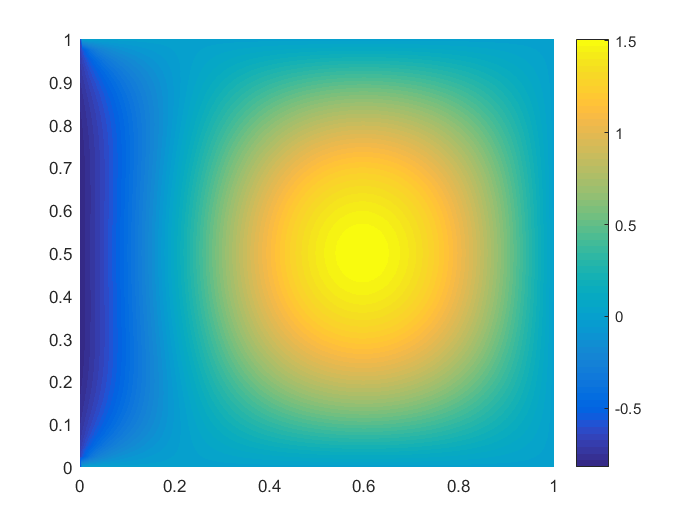}}
	\subfloat[Computed ($\varepsilon=10^{-1}$)\label{fig:1b}]{\includegraphics[scale=0.3]{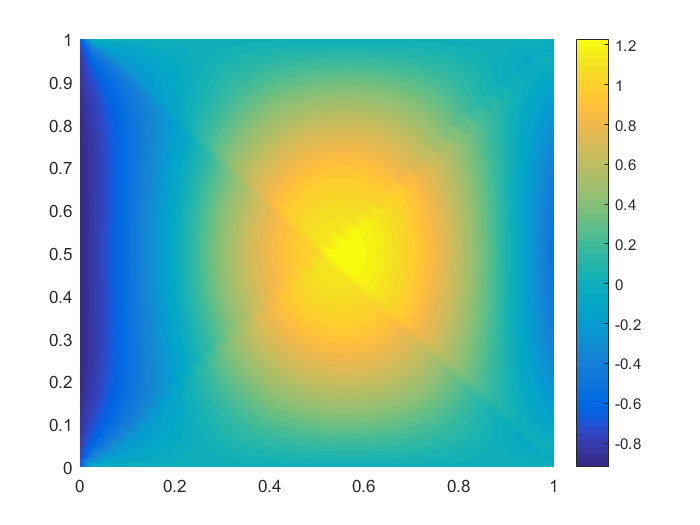}}
	\subfloat[Computed ($\varepsilon=10^{-2}$)\label{fig:1c}]{\includegraphics[scale=0.3]{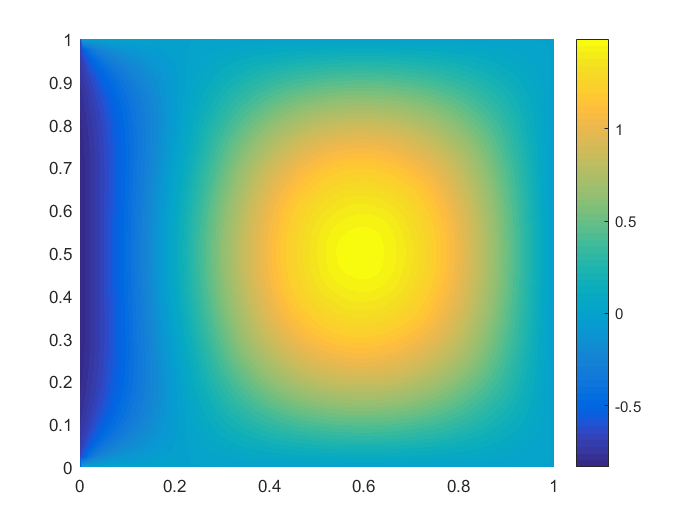}}
	
	\caption{Numerical results of Example 1 (low-frequency problem). (a) Graphical illustration of the true solution with $M=80$ and $N=40$. (b) and (c) Illustrations of the reconstructed solution with $\varepsilon = 10^{-1}$ and $\varepsilon = 10^{-2}$, respectively.\label{fig:1}}
	
\end{figure}

We begin this section by a numerical example with a low frequency profile. In this test, we particularly choose $k=5$ and
\[
u_{0}(y) = -e^{-2\left(0.5^4 + (y - 0.5)^4\right)} + 0.5^4 + (y - 0.5)^4.
\]
Such $k$ is suitable in the context of landmine detection; cf. e.g. \cite{Khoa2020}. In this low frequency profile, we observe numerically that the scheme works well with intermediate noise levels. Therefore, in this test the numerical results are taken into account with $\varepsilon = 10^{-1}$ and $10^{-2}$. Note that cf. \eqref{eta1}, $\eta_{\varepsilon}^2$ increases when $\varepsilon$ decreases, and it decreases when $M$ becomes larger. Henceforth, for our comparison purpose, to keep $\eta_{\varepsilon}$ unchanged  when decreasing $\varepsilon$, we need different values of $M$. In particular, when $\varepsilon=10^{-1}$, we take $M=40$, which gives $\eta_{\varepsilon}^2\approx 0.26$. When $\varepsilon= 10^{-2}$, we choose $M=80$.

Depicted in Figure \ref{fig:1} are the graphical illustrations of the true solution and its reconstructed with intermediate noise ($\varepsilon=10^{-1}$) and small noise ($\varepsilon = 10^{-2}$). When $\varepsilon$ is smaller, the computed solution is very close to the true one in terms of the value and, furthermore, the shape and location of the yellow circular protrusion; see Figures \ref{fig:1a} and \ref{fig:1c}. On the other hand, the relative error reduces from 34.703\% for $\varepsilon = 10^{-1}$ to 3.481\% for $\varepsilon = 10^{-2}$, which shows that the regularized solution obtained from solving \eqref{systemV} approximates well the true solution in this low frequency profile.

\subsection*{Example 2: Intermediate frequency}

\begin{figure}
	\subfloat[True\label{fig:3a}]{\includegraphics[scale=0.125]{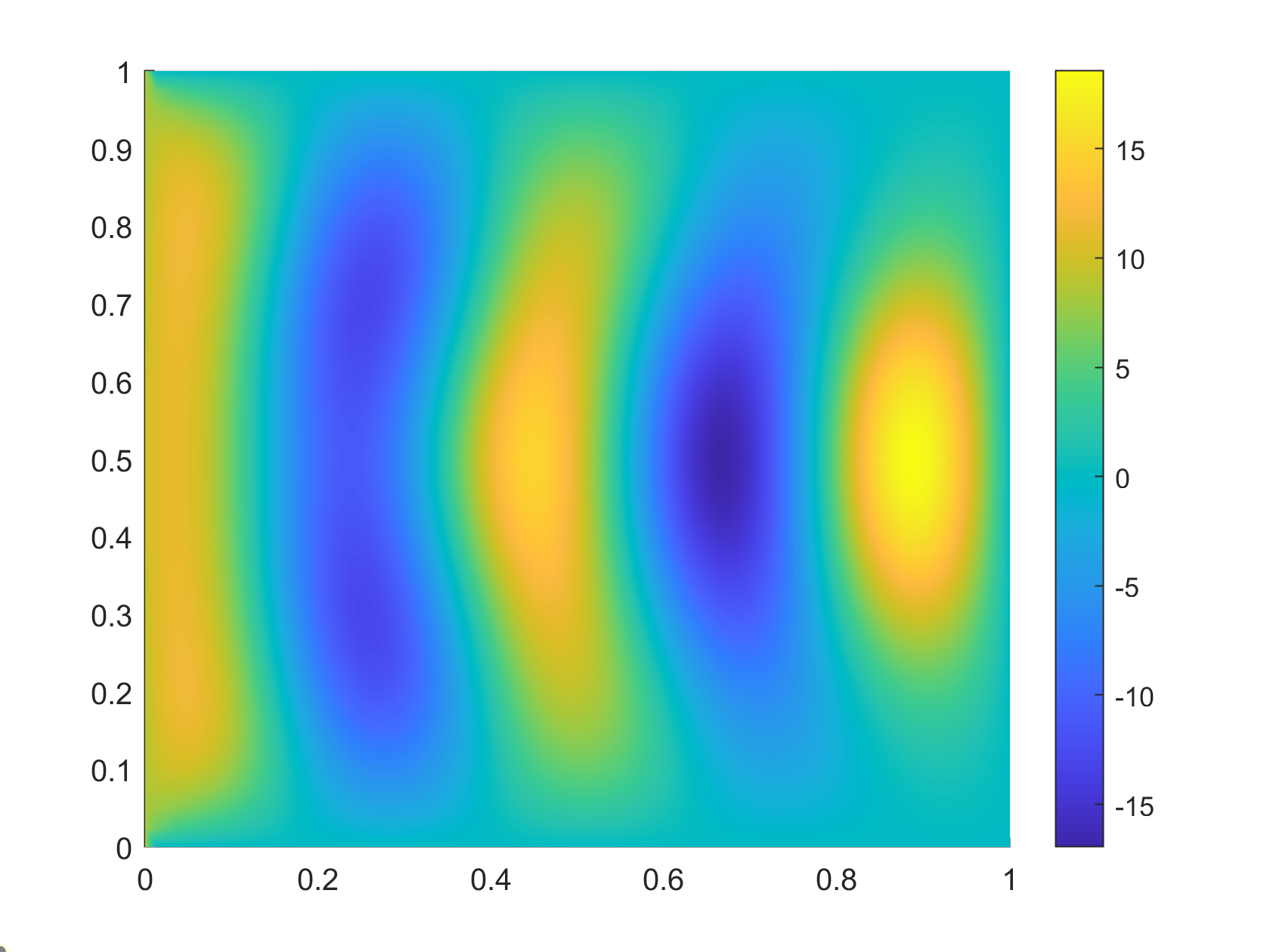}}
	\subfloat[Computed ($\varepsilon=10^{-1}$)\label{fig:3b}]{\includegraphics[scale=0.125]{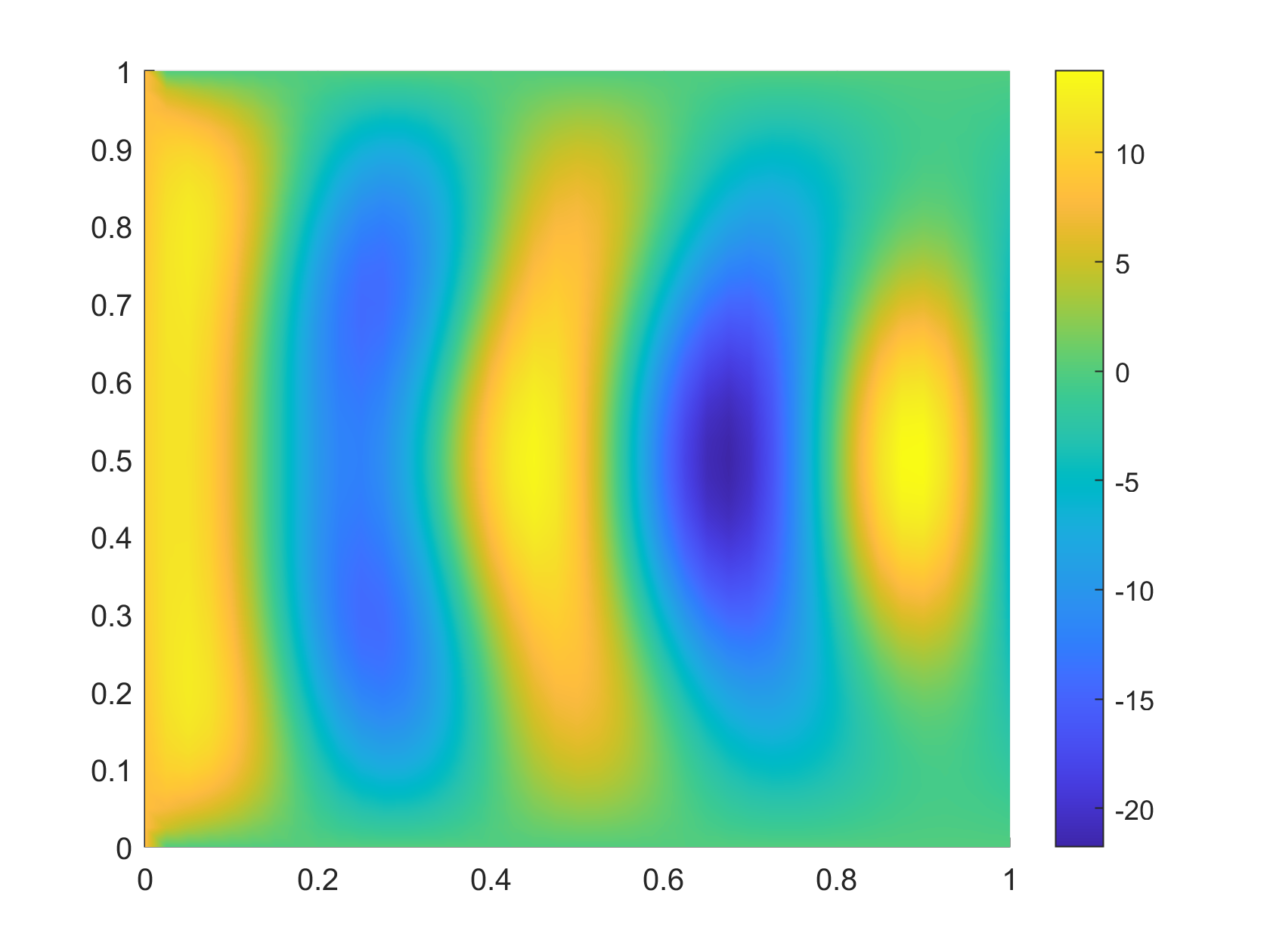}}
	\subfloat[Computed ($\varepsilon=10^{-2}$)\label{fig:3c}]{\includegraphics[scale=0.125]{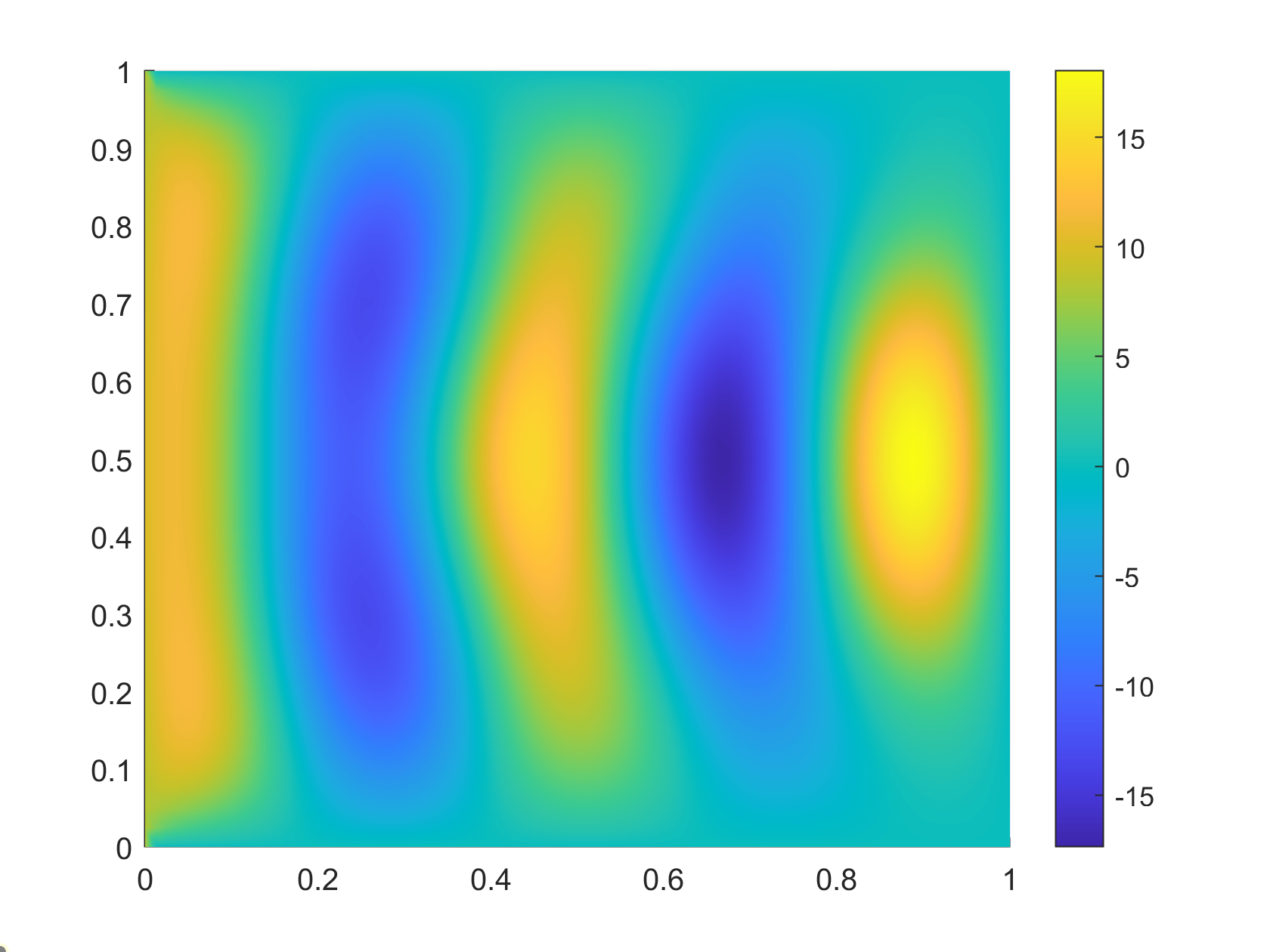}}
	
	\caption{Numerical results of Example 2 (intermediate-frequency problem). (a) Graphical illustration of the true solution with $M=80$ and $N=40$. (b) and (c) Illustrations of the reconstructed solution with $\varepsilon = 10^{-1}$ and $\varepsilon = 10^{-2}$, respectively.\label{fig:3}}
	
\end{figure}

In this test, we take into account an intermediate frequency problem with $k=15$. We choose that
\[
u_0(y) = \frac{1}{0.1 + 0.1(y-0.5)^2}.
\]
We verify the numerical performance of the iterative QR scheme with $\varepsilon = 10^{-1}$ and $\varepsilon=10^{-2}$. For each $\varepsilon$, we use the same parameters as taken in Example 1. Similar to the previous example, we observe numerically that the scheme reconstructs well the inclusions inside of the computational domain. The true solution and the reconstructed ones with  $\varepsilon = 10^{-1}$ and $\varepsilon=10^{-2}$ are reported in Figure \ref{fig:3}. Graphically, all yellow and blue inclusions are visible in Figure \ref{fig:3b} when the reconstruction is proceeded with $\varepsilon = 10^{-1}$ -- an intermediate noise. Their locations are also quite accurate, while only the values should be improved. Taking $\varepsilon$ smaller ($\varepsilon = 10^{-2}$), we can see the values in Figure \ref{fig:3c} are very close to the true ones in Figure \ref{fig:3a}. We also report that the relative error in this test reduces from 30.614\% (for $\varepsilon = 10^{-1}$) to 3.205\% (for $\varepsilon = 10^{-2}$).

\subsection*{Example 3: High frequency}

\begin{figure}
	\subfloat[True\label{fig:2a}]{\includegraphics[scale=0.3]{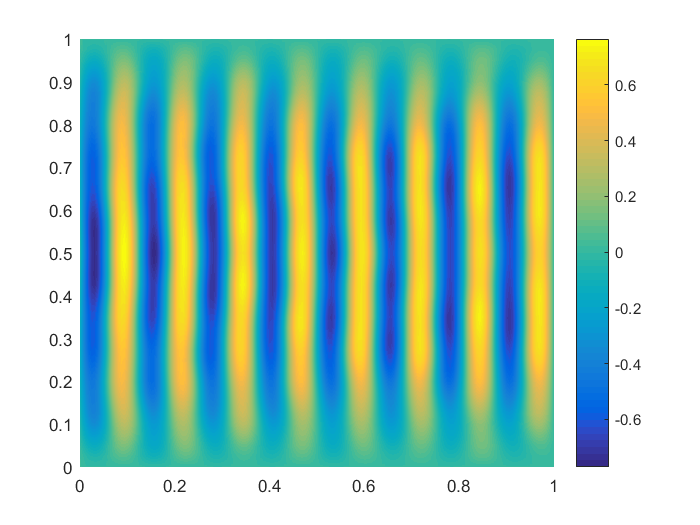}}
	\subfloat[Computed ($\varepsilon=10^{-2}$)\label{fig:2b}]{\includegraphics[scale=0.3]{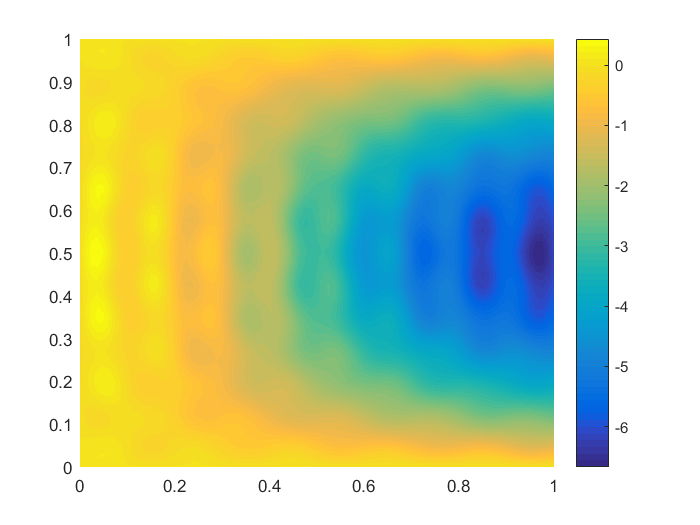}}
	\subfloat[Computed ($\varepsilon = 10^{-4}$)\label{fig:2c}]{\includegraphics[scale=0.3]{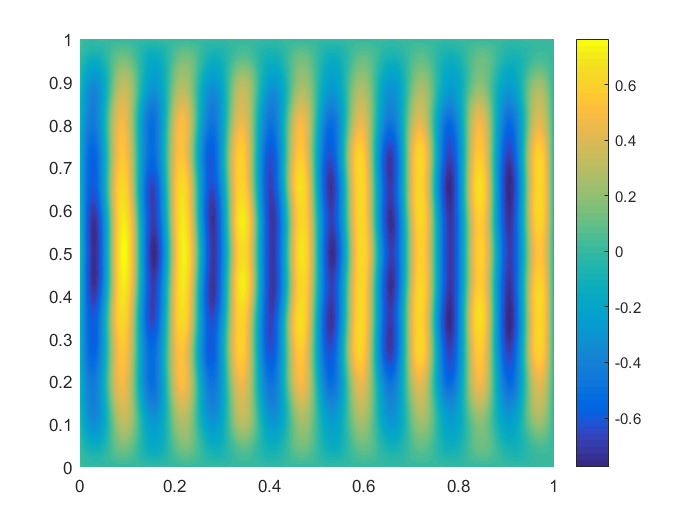}}
	
	\caption{Numerical results of Example 3 (high-frequency problem). (a) Graphical illustration of the true solution with $M=160$ and $N=40$. (b) and (c) Illustrations of the reconstructed solution with $\varepsilon = 10^{-2}$ and $\varepsilon = 10^{-4}$, respectively.\label{fig:2}}
	
\end{figure}

In this test, we consider a high frequency problem with $k = 50$ and
\[
u_0(y) = \frac{-\sin\left(7\sqrt{0.001 + (y-0.5)^2}\right)}{7\sqrt{1 + (y-0.5)^2}}.
\]
High frequency problems are usually challenging. Our numerical results for the well-posed problem \eqref{UU} of $U$ report that $M$ should be large enough for better resolution. In the regularized problem \eqref{systemV}, this also corresponds to choosing smaller values of $\varepsilon$. Thus, in this test, we report our numerical results with $\varepsilon=10^{-2}$ and $\varepsilon = 10^{-4}$. When $\varepsilon=10^{-2}$, we illustrate the reconstructed solution with $M=80$. When $\varepsilon = 10^{-4}$, we take $M=160$. Doing so ensures the same value of $\eta_{\varepsilon}^{2}$ discussed in the previous example.

Similar to Examples 1 and 2, we can see the reconstruction becomes better when $\varepsilon$ decreases, in this case, from $10^{-2}$ to $10^{-4}$; see Figure \ref{fig:2}. Especially, when $\varepsilon = 10^{-4}$, the computed solution, cf. Figure \ref{fig:2c}, shows exactly the same shape and location of all yellow bands in the true solution (Figure \ref{fig:2}). As can be seen from Figure \ref{fig:2b} for $\varepsilon=10^{-2}$, those bands are not even visible, and the value of the computed solution still undergoes the blow-up phenomenon due to the natural Hadamard instability. This graphical observation is not captured well in Example 1; see Figures \ref{fig:1b} and \ref{fig:1c}. This also explains why regularization of high frequency problems is rather challenging. We finally report that the relative error $E$ in this case reduces significantly from 1687.3\% to 7.212\%, when $\varepsilon$ decreases from $10^{-2}$ to $10^{-4}$.

\subsection*{Example 4: Extremely high frequency}

\begin{figure}
	\subfloat[True\label{fig:4a}]{\includegraphics[scale=0.125]{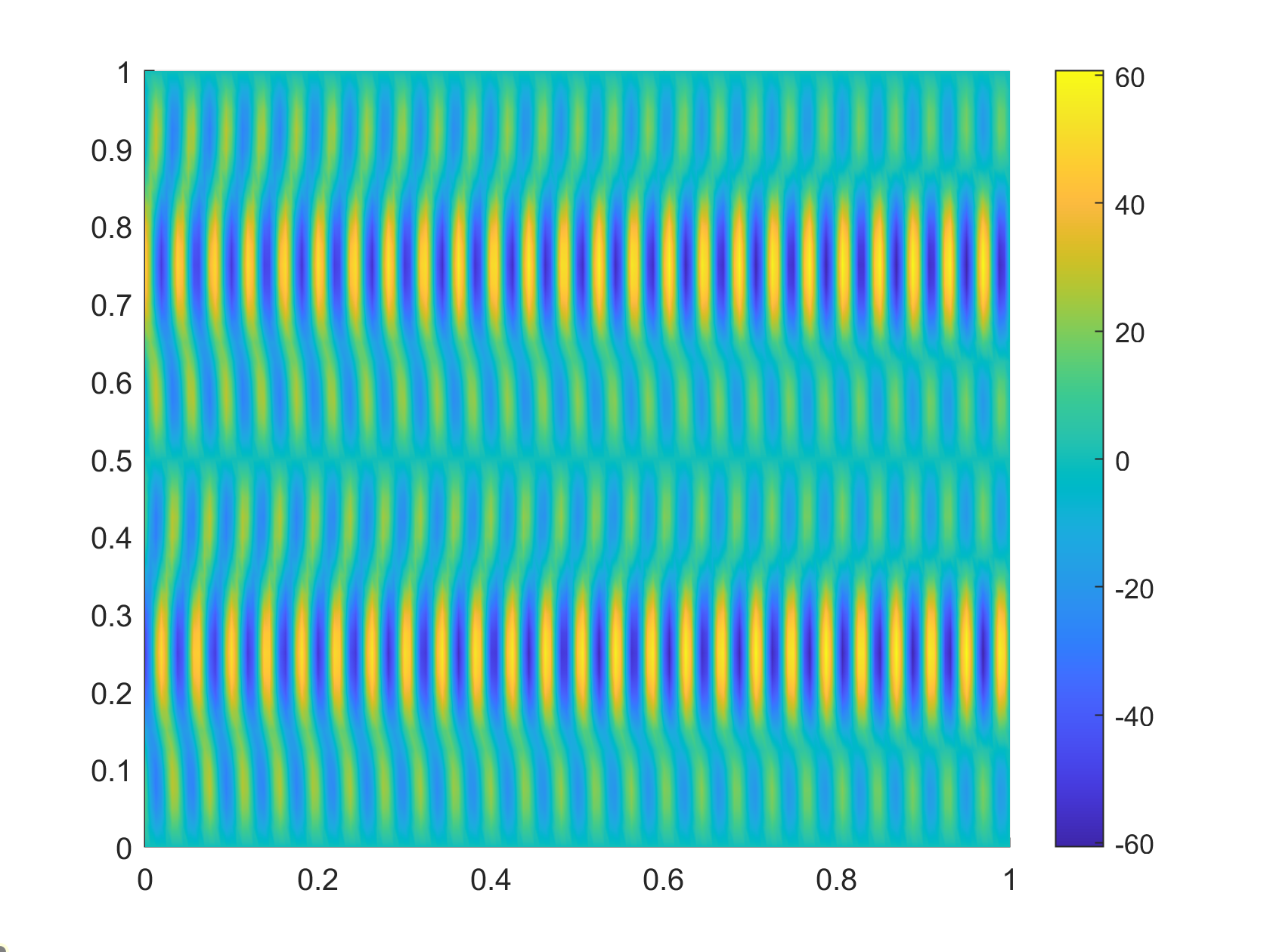}}
	\subfloat[Computed ($\varepsilon=10^{-2}$)\label{fig:4b}]{\includegraphics[scale=0.125]{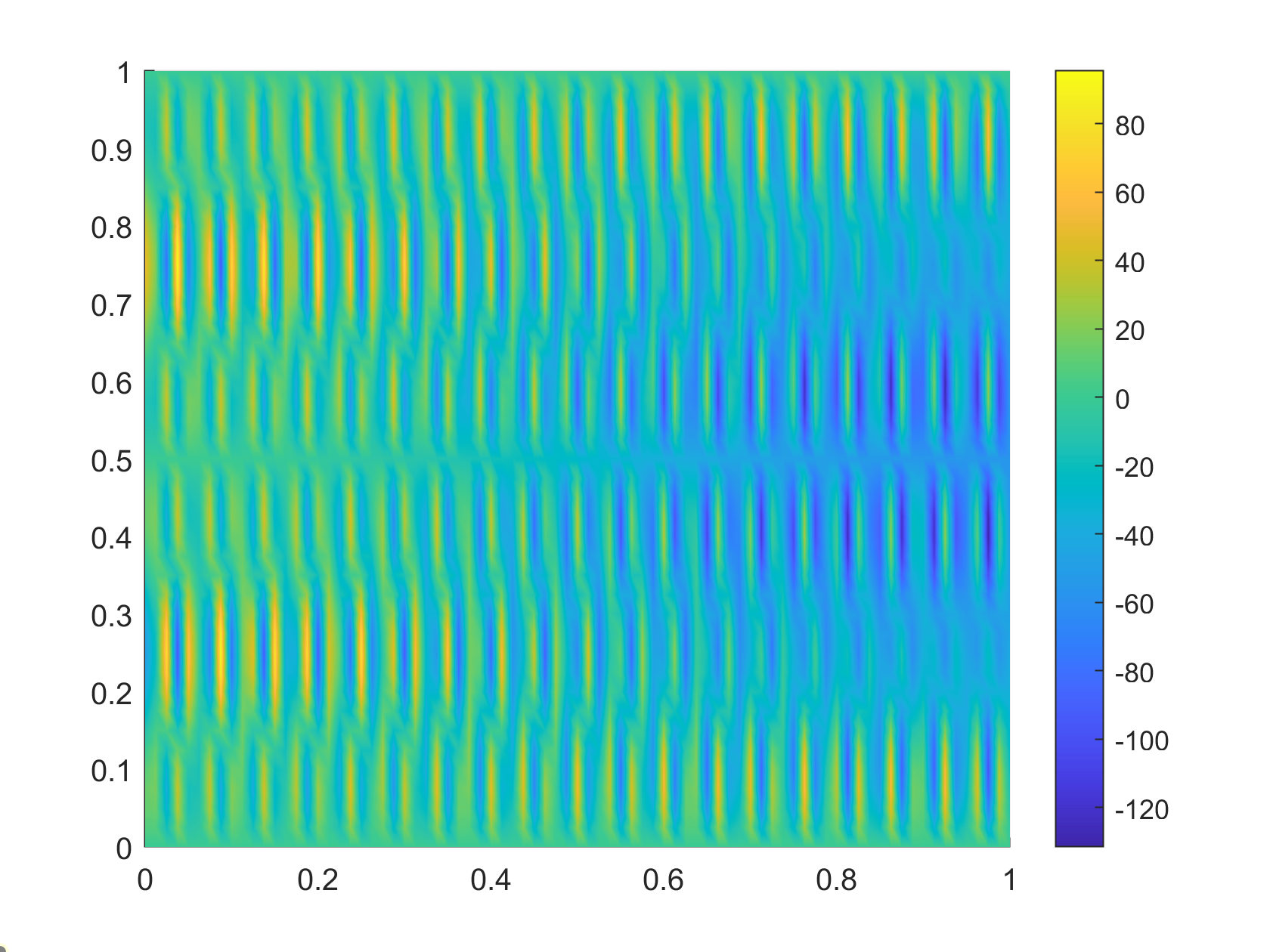}}
	\subfloat[Computed ($\varepsilon = 10^{-4}$)\label{fig:4c}]{\includegraphics[scale=0.125]{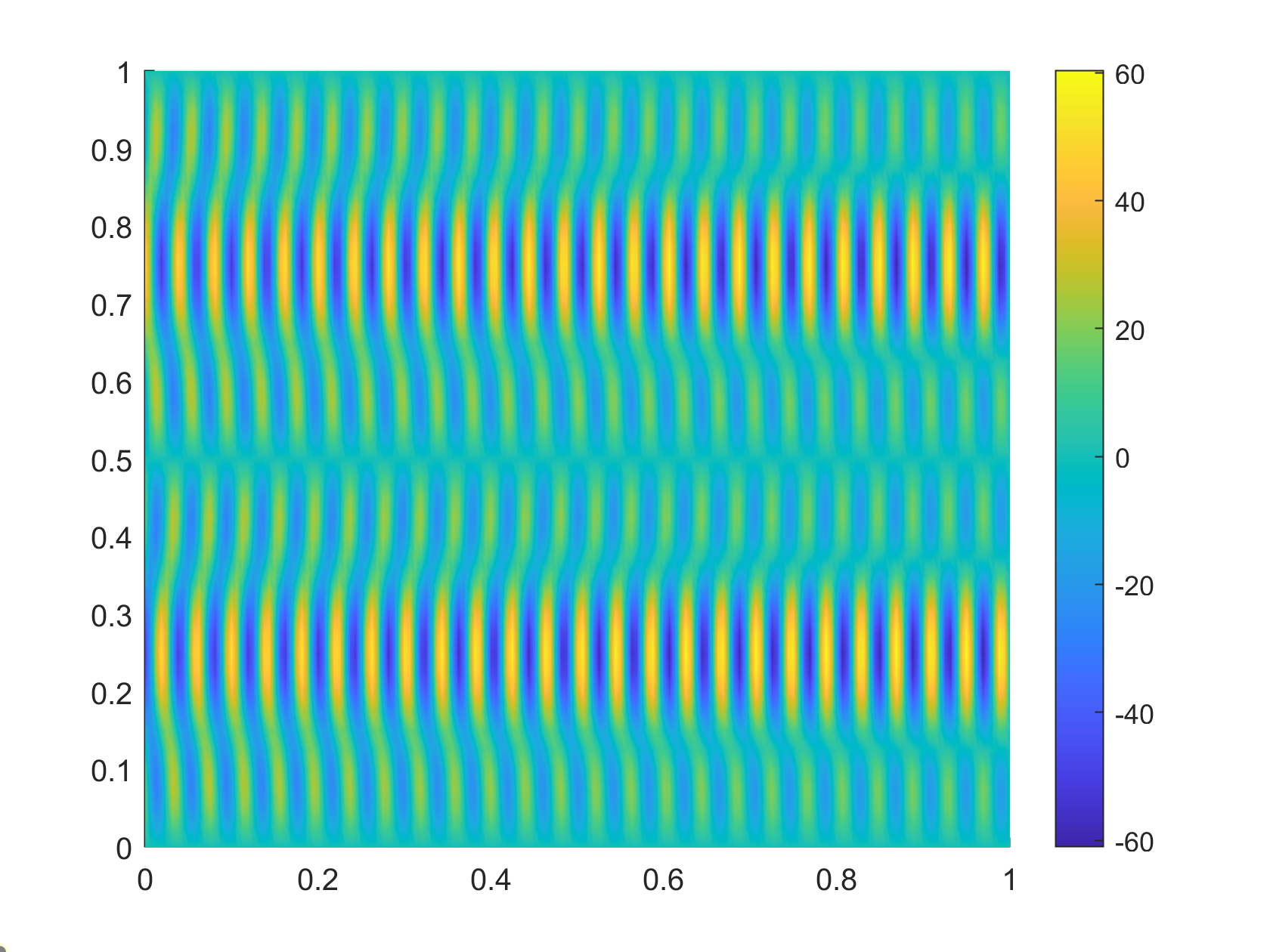}}
	
	\caption{Numerical results of Example 4 (extremely high-frequency problem). (a) Graphical illustration of the true solution with $M=160$ and $N=40$. (b) and (c) Illustrations of the reconstructed solution with $\varepsilon = 10^{-2}$ and $\varepsilon = 10^{-4}$, respectively.\label{fig:4}}
	
\end{figure}

In this last numerical example, we would like to see the performance of the scheme with a very large frequency. In particular, we choose $k=150$ and
\[
u_0(y) = 50\sin(2\pi y)\cos(4\pi y).
\]
Similar to Example 3, we obverse numerically that the scheme works with small noise levels. In this test, we verify the scheme when $\varepsilon = 10^{-2}$ and $10^{-4}$, and the same parameters are taken as in Example 3. In Figure \ref{fig:4b}, our reconstruction when $\varepsilon = 10^{-2}$ shows a slight accuracy in terms of shape and location of yellow bands within $x\in (0,0.2)$. When being far away from the ``initial'' data $u_0$, the reconstructed solution is pretty much inaccurate. This causes a huge relative error of 70.731\%. When $\varepsilon$ is down to $\varepsilon=10^{-4}$, this error reduces substantially to 1.153\%. The fine accuracy of the computed solution with $\varepsilon=10^{-4}$ can also be seen in Figure \ref{fig:4c}, compared with the true one in Figure \ref{fig:4a}. Note that as compared to the previous examples reconstructing a few inclusions, this last example indicates the efficiency of the method. It is, in fact, challenging if one wants to reconstruct several inclusions.

\section*{Acknowledgment}
V.A.K. thanks Prof. Dr. Roselyn Williams (Tallahassee, USA) for her support during the time V.A.K working at Florida A\&M University. N.D.T. acknowledges Dr. Nguyen Thanh Long for the wholehearted guidance during his study at University of Science, and thanks Mr. Pham Truong Hoang Nhan for helpful discussions.

\section*{Declarations}
The authors declare no competing interests.

\bibliography{bibtex}

\begin{thebibliography}{10}
\expandafter\ifx\csname url\endcsname\relax
  \def\url#1{\texttt{#1}}\fi
\expandafter\ifx\csname urlprefix\endcsname\relax\def\urlprefix{URL }\fi
\expandafter\ifx\csname href\endcsname\relax
  \def\href#1#2{#2} \def\path#1{#1}\fi

\bibitem{MaxBorn2019}
M.~Born, E.~Wolf, Principles of {O}ptics, Cambridge University Press, 2019.

\bibitem{Tuan2017}
N.~H. Tuan, V.~A. Khoa, M.~N. Minh, T.~Tran, Reconstruction of the electric
  field of the {H}elmholtz equation in three dimensions, Journal of
  Computational and Applied Mathematics 309 (2017) 56--78.
\newblock \href {https://doi.org/10.1016/j.cam.2016.05.021}
  {\path{doi:10.1016/j.cam.2016.05.021}}.

\bibitem{Klibanov2019a}
M.~V. Klibanov, D.-L. Nguyen, L.~H. Nguyen, A coefficient inverse problem with
  a single measurement of phaseless scattering data, {SIAM} Journal on Applied
  Mathematics 79~(1) (2019) 1--27.
\newblock \href {https://doi.org/10.1137/18m1168303}
  {\path{doi:10.1137/18m1168303}}.

\bibitem{Karimi2022}
M.~Karimi, Regularization of ill-posed problems involving constant-coefficient
  pseudo-differential operators, Inverse Problems 38~(5) (2022) 055001.
\newblock \href {https://doi.org/10.1088/1361-6420/ac5ac8}
  {\path{doi:10.1088/1361-6420/ac5ac8}}.

\bibitem{Leitao2000}
A.~Leit{\~{a}}o, An iterative method for solving elliptic cauchy problems,
  Numerical Functional Analysis and Optimization 21~(5-6) (2000) 715--742.
\newblock \href {https://doi.org/10.1080/01630560008816982}
  {\path{doi:10.1080/01630560008816982}}.

\bibitem{Qian2008}
Z.~Qian, C.-L. Fu, Z.-P. Li, Two regularization methods for a {C}auchy problem
  for the {L}aplace equation, Journal of Mathematical Analysis and Applications
  338~(1) (2008) 479--489.
\newblock \href {https://doi.org/10.1016/j.jmaa.2007.05.040}
  {\path{doi:10.1016/j.jmaa.2007.05.040}}.

\bibitem{Tuan2010a}
N.~H. Tuan, D.~D. Trong, P.~H. Quan, A note on a {C}auchy problem for the
  {L}aplace equation: {R}egularization and error estimates, Applied Mathematics
  and Computation 217~(7) (2010) 2913--2922.
\newblock \href {https://doi.org/10.1016/j.amc.2010.09.019}
  {\path{doi:10.1016/j.amc.2010.09.019}}.

\bibitem{Elden2009}
L.~Eld{\'{e}}n, V.~Simoncini, A numerical solution of a {C}auchy problem for an
  elliptic equation by {K}rylov subspaces, Inverse Problems 25~(6) (2009)
  065002.
\newblock \href {https://doi.org/10.1088/0266-5611/25/6/065002}
  {\path{doi:10.1088/0266-5611/25/6/065002}}.

\bibitem{Hao2000}
D.~N. H\`ao, D.~Lesnic, The {C}auchy problem for laplace's equation via the
  conjugate gradient method, {IMA} Journal of Applied Mathematics 65~(2) (2000)
  199--217.
\newblock \href {https://doi.org/10.1093/imamat/65.2.199}
  {\path{doi:10.1093/imamat/65.2.199}}.

\bibitem{Klibanov2015}
M.~V. Klibanov, Carleman estimates for the regularization of ill-posed {C}auchy
  problems, Applied Numerical Mathematics 94 (2015) 46--74.
\newblock \href {https://doi.org/10.1016/j.apnum.2015.02.003}
  {\path{doi:10.1016/j.apnum.2015.02.003}}.

\bibitem{Reinhardt1999}
H.-J. Reinhardt, H.~Han, D.~N. H\`ao, Stability and regularization of a
  discrete approximation to the {C}auchy problem for {L}aplace's equation,
  {SIAM} Journal on Numerical Analysis 36~(3) (1999) 890--905.
\newblock \href {https://doi.org/10.1137/s0036142997316955}
  {\path{doi:10.1137/s0036142997316955}}.

\bibitem{Falk1986}
R.~S. Falk, P.~B. Monk, Logarithmic convexity for discrete harmonic functions
  and the approximation of the {C}auchy problem for {P}oisson's equation,
  Mathematics of Computation 47~(175) (1986) 135.
\newblock \href {https://doi.org/10.2307/2008085} {\path{doi:10.2307/2008085}}.

\bibitem{Elden2005}
L.~Eld\'en, F.~Berntsson, A stability estimate for a {C}auchy problem for an
  elliptic partial differential equation, Inverse Problems 21~(5) (2005)
  1643--1653.
\newblock \href {https://doi.org/10.1088/0266-5611/21/5/008}
  {\path{doi:10.1088/0266-5611/21/5/008}}.

\bibitem{Karimi2017}
M.~Karimi, A.~Rezaee, Regularization of the {C}auchy problem for the
  {H}elmholtz equation by using {M}eyer wavelet, Journal of Computational and
  Applied Mathematics 320 (2017) 76--95.
\newblock \href {https://doi.org/10.1016/j.cam.2017.02.005}
  {\path{doi:10.1016/j.cam.2017.02.005}}.

\bibitem{Qiu2008}
C.-Y. Qiu, C.-L. Fu, Wavelets and regularization of the {C}auchy problem for
  the {L}aplace equation, Journal of Mathematical Analysis and Applications
  338~(2) (2008) 1440--1447.
\newblock \href {https://doi.org/10.1016/j.jmaa.2007.06.035}
  {\path{doi:10.1016/j.jmaa.2007.06.035}}.

\bibitem{LL67}
R.~Latt\`es, J.~L. Lions, M\'ethode de {Q}uasi-r\'eversibilit\'e et
  {A}pplications, Dunod, Paris, 1967.

\bibitem{Nguyen2019}
H.~T. Nguyen, V.~A. Khoa, V.~A. Vo, Analysis of a quasi-reversibility method
  for a terminal value quasi-linear parabolic problem with measurements, {SIAM}
  Journal on Mathematical Analysis 51~(1) (2019) 60--85.
\newblock \href {https://doi.org/10.1137/18m1174064}
  {\path{doi:10.1137/18m1174064}}.

\bibitem{Khoa2020a}
V.~A. Khoa, P.~T.~H. Nhan, Constructing a variational quasi-reversibility
  method for a {C}auchy problem for elliptic equations, Mathematical Methods in
  the Applied Sciences 44~(5) (2020) 3334--3355.
\newblock \href {https://doi.org/10.1002/mma.6945}
  {\path{doi:10.1002/mma.6945}}.

\bibitem{Klibanov2013}
M.~V. Klibanov, Carleman estimates for global uniqueness, stability and
  numerical methods for coefficient inverse problems, Journal of Inverse and
  Ill-Posed Problems 21~(4) (2013).
\newblock \href {https://doi.org/10.1515/jip-2012-0072}
  {\path{doi:10.1515/jip-2012-0072}}.

\bibitem{Khoa2020}
V.~A. Khoa, G.~W. Bidney, M.~V. Klibanov, L.~H. Nguyen, L.~H. Nguyen, A.~J.
  Sullivan, V.~N. Astratov, Convexification and experimental data for a 3{D}
  inverse scattering problem with the moving point source, Inverse Problems
  36~(8) (2020) 085007.
\newblock \href {https://doi.org/10.1088/1361-6420/ab95aa}
  {\path{doi:10.1088/1361-6420/ab95aa}}.

\bibitem{Klibanov2022}
M.~Klibanov, L.~H. Nguyen, H.~V. Tran, Numerical viscosity solutions to
  {H}amilton-{J}acobi equations via a {C}arleman estimate and the
  convexification method, Journal of Computational Physics 451 (2022) 110828.
\newblock \href {https://doi.org/10.1016/j.jcp.2021.110828}
  {\path{doi:10.1016/j.jcp.2021.110828}}.

\bibitem{Le2022}
T.~T. Le, M.~V. Klibanov, L.~H. Nguyen, A.~Sullivan, L.~Nguyen, Carleman
  contraction mapping for a 1{D} inverse scattering problem with experimental
  time-dependent data, Inverse Problems 38~(4) (2022) 045002.
\newblock \href {https://doi.org/10.1088/1361-6420/ac50b8}
  {\path{doi:10.1088/1361-6420/ac50b8}}.

\end{thebibliography}

\end{document}